\documentclass[11pt]{elsarticle}

\usepackage[ddmmyyyy]{datetime}
\usepackage{geometry}
\usepackage{lineno,hyperref}
\nolinenumbers
\usepackage{float}
\usepackage{amsfonts}
\usepackage{mathrsfs}
\usepackage{amssymb}
\usepackage{amsmath}
\usepackage{amsthm}
\usepackage{appendix}
\usepackage{xcolor}
\usepackage{graphicx}
\usepackage[font=small,labelfont=bf]{caption}
\usepackage{subcaption}
\usepackage{algpseudocode}
\usepackage{disN} 
\usepackage{disNmicros} 

\makeatletter
\def\ps@pprintTitle{%
 \let\@oddhead\@empty
 \let\@evenhead\@empty
 \def\@oddfoot{}%
 \let\@evenfoot\@oddfoot}
\makeatother

\begin{document}

\begin{frontmatter}

\title{Fully discrete needlet approximation on the sphere\tnoteref{tifn}}
\tnotetext[tifn]{This research was supported under the Australian Research Council's \emph{Discovery Project} DP120101816.
The first author was supported under the University International Postgraduate Award (UIPA) of UNSW Australia.}

\author{Yu Guang Wang\corref{coryuguang}}
\ead{yuguang.e.wang@gmail.com}
\author{Quoc T. Le Gia\corref{}}
\ead{qlegia@unsw.edu.au}
\author{Ian H. Sloan\corref{}}
\ead{i.sloan@unsw.edu.au}
\author{Robert S. Womersley\corref{}}
\ead{r.womersley@unsw.edu.au}

\cortext[coryuguang]{Corresponding author.}

\address{School of Mathematics and Statistics, UNSW Australia, Sydney, NSW, 2052, Australia}
%
%

\begin{abstract}
Spherical needlets are highly localized radial polynomials on the sphere $\sph{d}\subset \Rd[d+1]$, $d\ge 2$, with centers at the nodes of a suitable cubature rule. The original semidiscrete spherical needlet approximation of Narcowich, Petrushev and Ward is not computable, in that the needlet coefficients depend on inner product integrals. In this work we approximate these integrals by a second quadrature rule with an appropriate degree of precision, to construct a fully discrete needlet approximation. We prove that the resulting approximation is equivalent to filtered hyperinterpolation, that is to a filtered Fourier-Laplace series partial sum with inner products replaced by appropriate cubature sums. It follows that the $\mathbb{L}_{p}$-error of discrete needlet approximation of order $\neord$ for $1 \le p \le \infty$ and $s > d/p$ has for a function $f$ in the Sobolev space $\sob{p}{s}{d}$ the optimal rate of convergence in the sense of optimal recovery, namely $\bigo{}{2^{-\neord s}}$. Moreover, this is achieved with a filter function that is of smoothness class $C^{\floor{\frac{d+3}{2}}}$, in contrast to the usually assumed $C^{\infty}$. A numerical experiment for a class of functions in known Sobolev smoothness classes gives $\mathbb{L}_2$ errors for the fully discrete needlet approximation that are almost identical to those for the original semidiscrete needlet approximation. Another experiment uses needlets over the whole sphere for the lower levels together with high-level needlets with centers restricted to a local region. The resulting errors are reduced in the local region away from the boundary, indicating that local refinement in special regions is a promising strategy.

\end{abstract}


\end{frontmatter}


\section{Introduction}\label{sec:intro}
Spherical wavelets \cite{FrGeSc1998} have wide applications in areas such as signal
processing \cite{KhKeDuSaWiMc2013}, geography
\cite{Freeden1999,SiDaWi2006,SiLoBrDa2011} and cosmology
\cite{LaDu_etal2014,PiAmBaCaCoMa2008,ViMaBaSaCa2004}. The classical
continuous wavelets represent a complicated function by projecting it onto
different levels of a decomposition of the $\mathbb{L}_{2}$ function space
on the sphere. A projection, often called ``a detail of the function'',
becomes small rapidly as the level increases. This
multilevel decomposition proves very useful in solving many problems.

Narcowich et al. in recent work \cite{NaPeWa2006-1,NaPeWa2006-2} showed
that the details of the spherical wavelets may be further broken up into still finer
details, which are highly localized in space. This new decomposition of a
spherical function is said to be a needlet decomposition.

Needlet approximation in its original form is however not suitable for direct implementation as its needlet coefficients are integrals.
In this paper, we introduce a \emph{discrete} spherical needlet approximation
scheme by using spherical quadrature rules to approximate the inner
product integrals and establish its approximation error for functions in
Sobolev spaces on the sphere. Numerical experiments are carried out for
this fully discrete version of the spherical needlet approximation.

Before we describe spherical needlets and the discrete spherical needlet approximation we need
some definitions. For $d\geq2$, let $\mathbb{R}^{d+1}$ be the real
($d+1$)-dimensional Euclidean space with inner product $\PT{x}\cdot\PT{y}$
for $\PT{x},\PT{y}\in \REuc$ and Euclidean norm
$|\PT{x}|:=\sqrt{\PT{x}\cdot\PT{x}}$. Let
    $\sph{d}:=\{\PT{x}\in\REuc: |\PT{x}|=1\}$
denote the unit sphere of $\REuc$. The sphere $\sph{d}$ forms a compact
metric space, with the metric being the geodesic distance
  $\dist{\PT{x},\PT{y}}:=\arccos(\PT{x}\cdot\PT{y})$ for $\PT{x},\PT{y}\in \sph{d}$.

For $1\leq p \leq \infty$ let $\Lp{p}{d}=\Lp[,\sigma_{d}]{p}{d}$, endowed with the norm $\normL{\cdot}$, be the real $\mathbb{L}_{p}$-function space on $\sph{d}$ with respect to the normalised Riemann surface measure $\sigma_{d}$ on $\sph{d}$. For $p=2$, $\Lp{2}{d}$ forms a Hilbert space with inner product
  $\InnerL{f,g}:=\int_{\sph{d}}f(\PT{x})g(\PT{x})\IntDiff{x}$ for $f,g\in\Lp{2}{d}$.

A \emph{spherical harmonic} of degree $\ell$ on $\sph{d}$ is the restriction to $\sph{d}$ of a homogeneous and harmonic polynomial of total degree $\ell$ defined on $\REuc$.
Let $\mathcal{H}_{\ell}^{d}$ denote the set of all spherical harmonics of exact degree $\ell$ on $\sph{d}$. The dimension of the linear space $\mathcal{H}_{\ell}^{d}$ is
\begin{equation}\label{eq:dim.sph.harmon}
    Z(d,\ell):=(2\ell+d-1)\frac{\Gamma(\ell+d-1)}{\Gamma(d)\Gamma(\ell+1)}\asymp (\ell+1)^{d-1},
\end{equation}
where $\Gamma(\cdot)$ denotes the gamma function and $a_{\ell}\asymp
b_{\ell}$ means $c\:b_{\ell}\leq a_{\ell}\leq c' \:b_{\ell}$ for some
positive constants $c$, $c'$, and the asymptotic estimate uses
\cite[Eq.~5.11.12]{NIST:DLMF}. The linear span of
$\mathcal{H}_{\ell}^{d}$, $\ell=0,1,\dots,\nu$ forms the space
$\sphpo{\nu}$ of spherical polynomials of degree up to $\nu$. Let
$\Jcb{\ell}$ be the Jacobi polynomial of degree $\ell$ for
$\alpha,\beta>-1$. We denote the normalised Legendre or Gegenbauer
polynomial by
\begin{equation}\label{eq:normalised.Gegenbauer}
  \NGegen{\ell}(t):=\Jcb[\frac{d-2}{2},\frac{d-2}{2}]{\ell}(t)/\Jcb[\frac{d-2}{2},\frac{d-2}{2}]{\ell}(1).
\end{equation}

Given $N\geq1$, for $k=1,\dots,N$, let $\PT{x}_{k}$ be $N$ nodes on
$\sph{d}$ and let $w_{k}>0$ be corresponding weights. The set
$\{(w_{k},\PT{x}_{k}): k=1,\dots,N\}$ is a \emph{positive quadrature
(numerical integration) rule} exact for polynomials of degree up to  $\nu$
for some $\nu\geq0$ if
\begin{equation*}
  \int_{\sph{d}}\polysph(\PT{x})\IntDiff{x}=\sum_{k=1}^{N}w_{k}\:\polysph(\PT{x}_{k}),\quad \mbox{for~all~} \polysph\in\sphpo{\nu}.
\end{equation*}

\emph{Spherical needlets} \cite{NaPeWa2006-1,NaPeWa2006-2} are a type of localized polynomial on the sphere associated with a \emph{quadrature rule} and a \emph{filter}. Let $\mathbb{R}_{+}:=[0,+\infty)$.

\begin{definition}\label{def:fil.fil.ker}
A continuous compactly supported function $\fil: \mathbb{R}_{+}\to\mathbb{R}_{+}$ is said to be a filter. We will only consider a filter with support a subset of $[0,2]$.

A filtered kernel on $\sph{d}$ with filter $\fil$ is, for $T\in \mathbb{R}_{+}$,
\begin{equation}\label{eq:filter.sph.ker}
  \vdh{T,\fil}(\PT{x}\cdot\PT{y}):= \vdh[d]{T,\fil}(\PT{x}\cdot\PT{y}) :=\begin{cases}
  1, & 0\leq T<1,\\[1mm]
  \displaystyle\sum_{\ell=0}^{\infty}\fil\Bigl(\frac{\ell}{T}\Bigr)\:Z(d,\ell)\:\NGegen{\ell}(\PT{x}\cdot\PT{y}), & T\geq1.
  \end{cases}
\end{equation}
\end{definition}
  We may define a \emph{filtered approximation} $\Vdh{T,\fil}$ on $\Lp{1}{d}$, $T\geq0$ as an integral operator with the filtered kernel $\vdh{T,\fil}(\PT{x}\cdot\PT{y})$: for $f\in \Lp{1}{d}$,
\begin{equation}\label{eq:filter.sph.approx}
  \Vdh{T,\fil}(f;\PT{x}):= \Vdh[d]{T,\fil}(f;\PT{x}):= \InnerL{f,\vdh{T,\fil}(\PT{x}\cdot\cdot)}
  =\int_{\sph{d}} f(\PT{y})\:\vdh{T,\fil}(\PT{x}\cdot\PT{y})\:\IntDiff{y}.
\end{equation}
Note that for $T<1$ this is just the integral of $f$.

We are now ready to define a needlet, following \cite{NaPeWa2006-1} who used a $\CkR[\infty]$ filter and \cite{NaPeWa2006-2}. Let the \emph{needlet filter} $\fiN$ be a filter with specified smoothness $\kappa\geq 1$ (see Figure~\ref{fig:fiN} in Section~\ref{sec:numerics} for an example with $\kappa=5$) satisfying
\begin{subequations}\label{subeqs:fiN}
  \begin{align}
  &\fiN\in \CkR, \quad \supp \fiN = [1/2,2];\label{eq:fiN-a}\\[1mm]
  &\fiN(t)^{2} + \fiN(2t)^{2} = 1 \;\hbox{~if~} t\in [1/2,1].\label{eq:fiN-b}
  \end{align}
\end{subequations}
Condition \eqref{eq:fiN-b} is equivalent, given \eqref{eq:fiN-a}, to the following \emph{partition of unity} property for $h^{2}$,
\begin{equation*}
  \sum_{j=0}^{\infty} \fiN\Bigl(\frac{t}{{2^{j}}}\Bigr)^{2}=1, \quad t\geq1.
\end{equation*}

For $j=0,1,\dots$, we define the \emph{(spherical) needlet quadrature}
\begin{equation}\label{eq:QN}
\begin{array}{l}
    \{(\wN,\pN{jk}):k=1,\dots,N_{j}\}, \quad\wN>0, \; k=1,\dots,N_{j},\\
    \hbox{exact for polynomials of degree up to $2^{j+1}-1$}.
\end{array}
\end{equation}
A \emph{(spherical) needlet} $\needlet{jk}$, $k=1,\dots,N_{j}$ of order $j$ with needlet filter $\fiN$ and needlet quadrature \eqref{eq:QN} is then defined by
\begin{subequations}\label{subeqs:needlets}
\begin{equation}\label{eq:needlets-a}
  \needlet{jk}(\PT{x}) :=
  \sqrt{\wN}\: \vdh{2^{j-1},\fiN}(\PT{x}\cdot\PT{x}_{jk}),
\end{equation}
or equivalently,
\begin{align}\label{eq:needlets-b}
  \needlet{0k}(\PT{x}):=\sqrt{w_{0k}},\quad
  \needlet{jk}(\PT{x}) :=
  \displaystyle\sqrt{\wN}\: \sum_{\ell=0}^{\infty}\fiN\Bigl(\frac{\ell}{2^{j-1}}\Bigr)\:Z(d,\ell)\:\NGegen{\ell}(\PT{x}\cdot\pN{jk})
  \;\; \hbox{if~} j\geq1.
\end{align}
\end{subequations}
From \eqref{eq:fiN-a} we see that $\needlet{jk}$ is a polynomial of degree $2^{j}-1$. It is a band-limited polynomial, in that $\needlet{jk}$ is $\mathbb{L}_{2}$-orthogonal to all polynomials of degree $\leq 2^{j-2}$.

For $f\in \Lp{2}{d}$, the original \emph{(spherical) needlet approximation} with filter $\fiN$ and needlet quadrature \eqref{eq:QN} is defined (see \cite{NaPeWa2006-1}) by
\begin{equation}\label{eq:intro.needlet.approx}
  \neapx(f;\PT{x}):=\sum_{j=0}^{\neord}\sum_{k=1}^{N_{j}}\InnerL{f,\needlet{jk}} \needlet{jk}(\PT{x}),\quad \PT{x}\in\sph{d}.
\end{equation}
Note that $\neapx(f;\PT{x})$ is a polynomial of degree at most $2^{\neord}-1$ since $\needlet{jk}$ is a polynomial of degree $2^{j}-1$.
We shall call $\InnerL{f,\needlet{jk}}$ the \emph{semidiscrete (spherical) needlet coefficient} and $\neapx(f;\cdot)$ the \emph{semidiscrete (spherical) needlet approximation} to distinguish them from their fully discrete equivalents which we shall now introduce.

The \emph{discrete (spherical) needlet approximation} is defined by discretizing the inner-product integral
\begin{equation*}
  \InnerL{f,\needlet{jk}} = \int_{\sph{d}}f(\PT{y})\: \needlet{jk}(\PT{y})\: \IntDiff{y}
\end{equation*}
with another quadrature rule.
For $\nu\geq0$ and $N\geq1$, the \emph{discretization quadrature} rule is
\begin{equation}\label{eq:discrete.quadrature}
\begin{array}{l}
    \QH:=\gQ:=\{(\wH,\pH{i}):i=1,\dots,N\},\quad \wH>0,\; i = 1,\dots,N,\\
    \qquad\hbox{exact for polynomials of degree up to $\nu$}.
\end{array}
\end{equation}
Let $\ContiSph{}$ be the space of continuous functions on $\sph{d}$.
For $f,g\in \ContiSph{}$, given $\QH$ we define the \emph{discrete inner product} by
\begin{equation*}
  \InnerD{f,g} := \sum_{i=1}^{N} \wH \:f(\pH{i})\:g(\pH{i}).
\end{equation*}

The \emph{discrete (spherical) needlet coefficient} of $f$ for $\QH$ and $\needlet{jk}$ is $\InnerD{f,\needlet{jk}}$.
We then define the discrete needlet approximation of order $\neord$ by
\begin{equation}\label{eq:intro.discrete.needlet.approx}
  \disneapx(f;\PT{x}):=\sum_{j=0}^{\neord}\sum_{k=1}^{N_{j}}\InnerD{f,\needlet{jk}} \needlet{jk}(\PT{x}),\quad \PT{x}\in\sph{d}.
\end{equation}

Let $\fiN$ be a needlet filter with $\fis\ge\floor{\frac{d+3}{2}}$ and let $\sob{p}{s}{d}$ with $s\geq0$ and $1\le p\le \infty$ be a Sobolev space embedded in $\Lp{p}{d}$, see Section~\ref{subsec:RKHS} for the definition. In Theorem~\ref{thm:dis.needlets.err.Wp}, we prove as a special case that for $\nu=3\cdot 2^{\neord-1}-1$, i.e. $\QH=\QH[](N,3\cdot 2^{\neord-1}-1)$, the $\mathbb{L}_{p}$ error using the approximation \eqref{eq:intro.discrete.needlet.approx} for $f\in\sob{p}{s}{d}$ and $s>d/p$ has the convergence order $2^{-\neord s}$, i.e.
\begin{equation*}
    \normb{f-\disneapx(f)}{\Lp{p}{d}} \leq c\: 2^{-\neord s}\: \norm{f}{\sob{p}{s}{d}},\quad f\in \sob{p}{s}{d},
\end{equation*}
where the constant $c$ depends only on $d$, $s$, $\fiN$ and
$\fis$. This is consistent with the corresponding result for semidiscrete
needlet approximation, see \cite{NaPeWa2006-1} and
Theorem~\ref{thm:needlets.err.W}:
\begin{equation*}
  \normb{f-\neapx(f)}{\Lp{p}{d}}\leq c \: 2^{-\neord s} \norm{f}{\sob{p}{s}{d}}.
\end{equation*}

In Sections~\ref{sec:filter.semidiscrete.needlet} and \ref{sec:discrete.needlet.approx} we establish the connection to wavelets, and
prove that the needlet approximation is equivalent to a filtered approximation and that the discrete needlet approximation
is equivalent to \emph{filtered hyperinterpolation} \cite{SlWo2012} --- a fully discrete version of the filtered approximation.
These connections draw attention to the fact that the discrete needlet
approximation considered in the present paper is not of itself new: what
we have done is to express the filtered hyperinterpolation approximation
in terms of a \emph{frame} $\{\needlet{jk}\}$ of the polynomial space where the frame has strong localization properties. The
benefit  will become apparent, however, if we take advantage of the local
nature of the approximation to carry out local refinement. We make a
preliminary study of local refinement of this kind in a numerical
experiment in Section~\ref{sec:numerics}, though in this paper we
do not develop the local theory. Rather, our main emphasis in this paper
is on establishing the necessary theoretical tools for the discrete
needlet approximation, on demonstrating the precise relationship between
the various approximations, and on obtaining a global error analysis for
$f$ in Sobolev spaces.

We note that Mhaskar \cite{Mhaskar2005,Mhaskar2006} proposed a full-discrete filtered polynomial approximation which is equivalent to filtered hyperinterpolation. A central assumption in \cite{Mhaskar2005,Mhaskar2006}, in addition to polynomial exactness, is that a Marcinkiewicz-Zygmund (M-Z) inequality is satisfied. Quadrature rules with positive weights and polynomial exactness automatically satisfy an M-Z inequality (see Dai \cite[Theorem~2.1]{Dai2006} and Mhaskar \cite[Theorem~3.3]{Mhaskar2006}). However, neither decomposition of wavelets into needlets nor numerical implementation were studied in \cite{Mhaskar2005,Mhaskar2006}.

Semidiscrete needlets are studied in other spaces, such as $\Rd$, diffusion spaces and Dirichlet spaces, see \cite{CoKePe2012,IvPeXu2010,KeNiPi2012,MaMh2008}. Quadrature rules and M-Z inequalities for these spaces are studied in \cite{FiMh2010,FiMh2011}. Establishing fully discrete needlet approximations on these spaces could be of great interest and will have wide applications.

The paper is organised as follows. Section~\ref{sec:pre} gives necessary
preparations. Section~\ref{sec:filter.semidiscrete.needlet} studies the
semidiscrete needlet approximation and its $\mathbb{L}_{p}$ approximation
errors for $f$ in Sobolev spaces on $\sph{d}$, and its connection with the
filtered approximation and continuous wavelets. In
Section~\ref{sec:discrete.needlet.approx}, we discuss the fully discrete
needlet approximation and prove its approximation error for $f\in\sob{p}{s}{d}$
and exploit its relation to the filtered hyperinterpolation approximation
and discrete wavelets. In Section~\ref{sec:numerics}, we give
numerical examples of needlets and then some numerical experiments.
Section~\ref{sec:proofs.sec.filter.semidiscrete.needlet} gives the proofs for
the results in Section~\ref{sec:filter.semidiscrete.needlet}.

\section{Preliminaries}\label{sec:pre}
The big $\mathcal{O}$ notation $a(L)=\bigo{\alpha}{b(L)}$ means that there exists a constant $c_{\alpha}>0$ depending only on $\alpha$ such that $|a(L)|\leq c_{\alpha}|b(L)|$.
The ceiling function $\left\lceil x\right\rceil$ is the smallest integer at least $x$ and the floor function $\left\lfloor x\right\rfloor$ is the largest integer at most $x$. For integer $k\geq0$ and real $a\geq k$, let
    ${a\choose k}:=\frac{a(a-1)\cdots (a-k+1)}{k!}=\frac{\Gamma(a+1)}{\Gamma(a-k)\Gamma(k+1)}$
be the extended binomial coefficient. We use ``$T$'' as a non-negative real number  and use ``$L$'' as a power of $2$, specially $L:=2^{\neord-1}$ for some $\neord\ge1$, to remove any ambiguity about the meaning of the symbol $L$.

\subsection{$\mathbb{L}_p$ space on sphere}
Let $\Lp{p}{d}=\Lp[,\sigma_{d}]{p}{d}$ be the $\mathbb{L}_{p}$-function space with respect to the normalised Riemann surface measure $\sigma=\sigma_{d}$ on $\sph{d}$, endowed with the $\mathbb{L}_{p}$ norm
\begin{align*}
  \norm{f}{\Lp{p}{d}} &:= \left\{\int_{\sph{d}}|f(\PT{x})|^{p}\mathrm{d}\sigma_{d}(\PT{x})\right\}^{1/p},\quad f\in \Lp{p}{d},\;\; 1\leq p<\infty;\\[0.1cm]
  \norm{f}{\Lp{\infty}{d}} &:= \sup_{\PT{x}\in \sph{d}}|f(\PT{x})|,\quad f\in \Lp{\infty}{d}\cap \ContiSph{}.
\end{align*}

Since each pair of $\mathcal{H}_{\ell}^{d}$, $\mathcal{H}_{\ell'}^{d}$ for $\ell> \ell'\geq0$ is $\mathbb{L}_{2}$-orthogonal, it follows that $\sphpo{\nu}$ is the direct sum of $\mathcal{H}_{\ell}^{d}$, i.e. $\sphpo{\nu}=\bigoplus_{\ell=0}^{\nu} \mathcal{H}_{\ell}^{d}$. Moreover, the direct sum $\bigoplus_{\ell=0}^{\infty} \mathcal{H}_{\ell}^{d}$ is dense in $\Lp{p}{d}$, see e.g. \cite[Ch.1]{WaLi2006}. Each member of $\mathcal{H}_{\ell}^{d}$ is an eigenfunction of the negative Laplace-Beltrami operator $-\LBo$ on the sphere $\sph{d}$, with eigenvalue
\begin{equation}\label{eq:eigenvalue}
  \lambda_{\ell}=\lambda_{\ell}^{(d)}:=\ell(\ell+d-1).
\end{equation}

A \emph{zonal function} is a function $K:\sph{d}\times\sph{d}\rightarrow \mathbb{R}$ that depends only on the inner product of the arguments, i.e. $K(\PT{x},\PT{y})=g(\PT{x}\cdot\PT{y})$,\: $\PT{x},\PT{y}\in \sph{d}$, for some function $g:[-1,1]\to \mathbb{R}$. Let $\NGegen{\ell}(\PT{x}\cdot \PT{y})$ be the zonal function given by \eqref{eq:normalised.Gegenbauer}.
From \cite[Theorem~7.32.1, p.~168]{Szego1975},
\begin{equation}\label{eq:NGegen.max}
  \bigl|\NGegen{\ell}(\PT{x}\cdot\PT{y})\bigr|\leq 1.
\end{equation}

Let $\{Y_{\ell,m}=Y_{\ell,m}^{(d)}: \ell\geq0,\; m=1,\dots,Z(d,\ell)\}$ denote an \emph{orthonormal basis} of the space $\Lp{2}{d}$.
The normalised Legendre polynomial $\NGegen{\ell}(\PT{x}\cdot\PT{y})$ satisfies the \emph{addition theorem}
\begin{equation}\label{eq:addition.theorem}
  \sum_{m=1}^{Z(d,\ell)}Y_{\ell,m}(\PT{x})Y_{\ell,m}(\PT{y})=Z(d,\ell)\NGegen{\ell}(\PT{x}\cdot \PT{y}).
\end{equation}

This with the orthogonality of $Y_{\ell,m}$ together gives the orthogonality of $\NGegen{\ell}$ and $\NGegen{\ell'}$:
\begin{align}\label{eq:orthogonal.NGegen}
\InnerL{Z(d,\ell)\NGegen{\ell}(\PT{x}\cdot\cdot),Z(d,\ell')\NGegen{\ell'}(\PT{y}\cdot\cdot)}
    =\left\{\begin{array}{ll}Z(d,\ell)\NGegen{\ell}(\PT{x}\cdot\PT{y}), &\ell=\ell',\\
                                0, &\ell\neq\ell'.
  \end{array}\right.
\end{align}

Let $v(\PT{x}\cdot\PT{y})$ and $g(\PT{x}\cdot\PT{y})$ be two zonal functions of the form
\begin{equation*}
    v(\PT{x}\cdot\PT{y}) := \sum_{\ell=0}^{\infty} a_{\ell}\: Z(d,\ell)\: \NGegen{\ell}(\PT{x}\cdot\PT{y}),\quad
    g(\PT{x}\cdot\PT{y}) := \sum_{\ell=0}^{\infty} b_{\ell}\: Z(d,\ell)\: \NGegen{\ell}(\PT{x}\cdot\PT{y}).
\end{equation*}
Then \eqref{eq:orthogonal.NGegen} gives, for $\PT{x}$, $\PT{z} \in \sph{d}$,
\begin{equation}\label{eq:innerL.zonal.kers}
  \InnerLb{v(\PT{x}\cdot\cdot),g(\PT{z}\cdot\cdot)} = \int_{\sph{d}} v(\PT{x}\cdot\PT{y}) \:g(\PT{z}\cdot\PT{y})\:\IntDiff{y} = \sum_{\ell=0}^{\infty} a_{\ell}\:b_{\ell}\: Z(d,\ell) \: \NGegen{\ell}(\PT{x}\cdot\PT{z}).
\end{equation}

\subsection{Sobolev spaces on the sphere}
Let $s\in \mathbb{R}_{+}$. We define
\begin{equation}\label{eq:b.ls}
  b_{\ell}^{(s)}:=(1+\lambda_{\ell})^{s/2}\asymp (1+\ell)^{s},
\end{equation}
where $\lambda_{\ell}$ is given by \eqref{eq:eigenvalue}. For $\ell\geq0$, $m=1,\dots,Z(d,\ell)$, let
\begin{equation*}
 \Fcoe{f} := \InnerL{f,Y_{\ell,m}} := \int_{\sph{d}}f(\PT{x})Y_{\ell,m}(\PT{x}) \IntDiff{x}
\end{equation*}
be the Fourier coefficients of $f\in\Lp{1}{d}$.

The \emph{generalised Sobolev space} $\sob{p}{s}{d}$ with $s>0$ may be defined as the set of all functions $f\in \Lp{p}{d}$ satisfying
  $\sum_{\ell=0}^{\infty}b_{\ell}^{(s)}\sum_{m=1}^{Z(d,\ell)} \Fcoe{f} Y_{\ell,m}\: \in\: \Lp{p}{d}$.
The Sobolev space $\sob{p}{s}{d}$ is a Banach space with norm
\begin{equation}\label{eq:sob.norm}
  \norm{f}{\sob{p}{s}{d}}:=\normB{\sum_{\ell=0}^{\infty}b_{\ell}^{(s)}\sum_{m=1}^{Z(d,\ell)} \Fcoe{f} Y_{\ell,m}}{\Lp{p}{d}}.
\end{equation}
We have the following two embedding lemmas for $\sob{p}{s}{d}$, see \cite{Kamzolov1982} and also {\relax\cite[Eq.~14,~p.~420]{Hesse2006}}. Given a nonnegative integer $\fis$, let $\ContiSph{\fis}$ denote the set of all $\fis$ times continuously differentiable functions on $\sph{d}$.
\begin{lemma}[Continuous embedding into $\ContiSph{\fis}$] Let $d\geq2$. Given a non-negative integer $\fis$, the Sobolev space $\sob{p}{s}{d}$ is continuously embedded into $\ContiSph{\fis}$ if $s>\fis+d/p$.
\end{lemma}

\begin{lemma} Let $d\geq2$. For $0<s\leq s'<\infty$ and $1\leq p\leq p'<\infty$, $\sob{p'}{s'}{d}$ is continuously embedded into $\sob{p}{s}{d}$.
\end{lemma}

\subsection{Reproducing kernel Hilbert spaces}\label{subsec:RKHS}
When $p=2$ and $s>d/2$ the space $\sob{p}{s}{d}$ becomes a reproducing kernel Hilbert space. For brevity, we write
    $\sobH{s} := \sob{2}{s}{d}$.
The inner product in $\sobH{s}$ is
  $\InnerH{f,g} := \sum_{\ell=0}^{\infty}\sum_{m=1}^{Z(d,\ell)}b^{(2s)}_{\ell}\:\Fcoe{f}\:\Fcoe{g}$.
By \eqref{eq:sob.norm} and $\mathbb{L}_2$-orthogonality of the $Y_{\ell,m}$,
  $\norm{f}{\sobH{s}}=\left(\sum_{\ell=0}^{\infty}\sum_{m=1}^{Z(d,\ell)} b_{\ell}^{(2s)} \bigl|\Fcoe{f}\bigr|^{2}\right)^{1/2}$.

For $s>d/2$, each $\sobH{s}$ has associated with it a unique kernel
$\RK[\PT{x},\PT{y}]{s}$ satisfying, for $\PT{x}\in\sph{d}$ and $f \in
\sobH{s}$, see e.g. \cite[Section~2.4]{BrSaSlWo2014},
\begin{equation}\label{eq:reproduing.property}
  \RK[\PT{x},\cdot]{s}\in \sobH{s},\quad \InnerH{f,\RK[\PT{x},\cdot]{s}} = f(\PT{x}).
\end{equation}
From \eqref{eq:reproduing.property}, we have
\begin{equation}\label{eq:RK.reproduing.property}
  \InnerH{\RK[\PT{x},\cdot]{s},\RK[\PT{y},\cdot]{s}} = \RK[\PT{x},\PT{y}]{s}, \;\;\PT{x},\PT{y}\in\sph{d}.
\end{equation}
The kernel $\RK[\PT{x},\PT{y}]{s}$ is said to be the \emph{reproducing kernel} for $\sobH{s}$. It is moreover a zonal kernel, taking the explicit form
\begin{equation}\label{eq:RK.sph}
  \RK{s} := \sum_{\ell=0}^{\infty} b^{(-2s)}_{\ell} \: Z(d,\ell)\: \NGegen{\ell}(\PT{x}\cdot\PT{y})
  = \sum_{\ell=0}^{\infty}\sum_{m=1}^{Z(d,\ell)} b^{(-2s)}_{\ell} \: Y_{\ell,m}(\PT{x})\:Y_{\ell,m}(\PT{y}).
\end{equation}

\section{Filtered operators, needlets and wavelets}\label{sec:filter.semidiscrete.needlet}
In this section, we study the properties of the filtered kernel, needlets and wavelets, and their relationships.

\subsection{Semidiscrete needlets and continuous wavelets}
We now point out the relation between spherical needlets and spherical wavelet decompositions. Let $\fiN$ be a needlet filter
satisfying \eqref{subeqs:fiN}. Obviously the semidiscrete needlet
approximation \eqref{eq:intro.needlet.approx} can be
written, for $f\in \Lp{1}{d}$ and $\PT{x}\in\sph{d}$, as
\begin{equation*}
  \neapx(f;\PT{x}) = \sum_{j=0}^{\neord}\parsum{j}(f;\PT{x}),
\end{equation*}
where $\parsum{j}(f)$ is the contribution to the semidiscrete needlet approximation for level $j$:
\begin{equation}\label{eq:Uj}
  \parsum{j}(f;\PT{x}):=\sum_{k=1}^{N_{j}}\InnerL{f,\needlet{jk}}\needlet{jk}(\PT{x}).
\end{equation}
In the language of wavelets, we may consider $\parsum{j}(f;\PT{x})$ to be the level-$j$ ``detail'' of the approximation $\neapx(f)$.

Needlets have a close relation to filtered polynomial approximations. At
the heart of this relationship is the following expression, due to
\cite{NaPeWa2006-1}, and stated formally in
Theorem~\ref{thm:needlets.vs.filter.sph.ker} below: if $\needlet{jk}$
denotes the needlets of order $j\geq0$ with needlet filter $\fiN$ and
needlet quadrature \eqref{eq:QN}, then
\begin{align}\label{eq:intro.needlets.vs.filter.sph.ker.fiN}
  \sum_{k=1}^{N_{j}}\needlet{jk}(\PT{x})\:\needlet{jk}(\PT{y})
  &=  \vdh{2^{j-1},\fiN^{2}}(\PT{x}\cdot\PT{y}),
\end{align}
in which the filter on the right-hand side, it should be noted, is $\fiN^{2}$, the square of the needlet filter. This means that the level-$j$ contribution to the semidiscrete needlet approximation can be written, using \eqref{eq:Uj}, as
\begin{equation*}
  \parsum{j}(f;\PT{x}):=\int_{\sph{d}}f(\PT{x})\:\vdh{2^{j-1},\fiN^{2}}(\PT{x}\cdot\PT{y})\:\IntDiff{y}.
\end{equation*}

To obtain the full semidiscrete needlet approximation we need to sum over $j$. For this purpose we introduce a new filter $\fiH$ related to the needlet filter $\fiN$:
\begin{equation}\label{eq:fiH}
\fiH(t):=\left\{\begin{array}{ll}
1, & 0\leq t<1,\\
\fiN(t)^{2}, & t\geq1,
\end{array}\right.
\end{equation}
and use the properties
\begin{equation}\label{eq:fiH.fiN}
  \supp\fiH\subset[0,2],\quad \fiH\Bigl(\frac{t}{2^{\neord}}\Bigr)=\sum_{j=0}^{\neord}\fiN\Bigl(\frac{t}{2^{j}}\Bigr)^{2},\quad t\geq1,\;\neord\in\mathbb{Z}_{+},
\end{equation}
which are easy consequences of \eqref{subeqs:fiN}. We note that this implies $\fiH\in\CkR$ given $\fiN\in\CkR$.
It then follows that
    $\sum_{j=0}^{\neord} \vdh{2^{j-1},\fiN^{2}}(\PT{x}\cdot\PT{y}) = \vdh{2^{\neord-1},\fiH}(\PT{x}\cdot\PT{y}),\; \neord=0,1,\dots$,
and as a result the semidiscrete needlet approximation can be expressed as
\begin{equation*}
  \neapx(f;\PT{x}) =  \int_{\sph{d}} f(\PT{x})\:\vdh{2^{\neord-1},\fiH}(\PT{x}\cdot\PT{y})\:\IntDiff{y}
\end{equation*}
and for $\neord\ge1$ this equals $\Vdh{L,\fiH}(f)$, where $L:=2^{\neord-1}$.

\subsection{Filtered operator and its kernel}\label{subsec:filter.sph.operator.kernel}
Recall the definition \eqref{eq:filter.sph.ker} of a filtered kernel. The convolution of two filtered kernels is also a filtered kernel. In particular, we have
\begin{proposition}\label{prop:inner.prod.filter.sph.kers} Let $d\geq2$ and let $\fil$ be a filter. Then for $T\geq0$ and $\PT{x},\PT{z}\in\sph{d}$,
    \begin{equation}\label{eq:inner.prod.filter.sph.kers}
      \InnerL{\vdh{T,\fil}(\PT{x}\cdot\cdot),\vdh{T,\fil}(\PT{z}\cdot\cdot)}
      :=\int_{\sph{d}}{\vdh{T,\fil}(\PT{x}\cdot\PT{y})\:\vdh{T,\fil}(\PT{z}\cdot\PT{y})}\: \IntDiff{y}
      = \vdh{T,\fil^2}(\PT{x}\cdot\PT{z}).
    \end{equation}
\end{proposition}
\begin{proof} For $0\leq T <1$, by \eqref{eq:filter.sph.ker}, both sides of \eqref{eq:inner.prod.filter.sph.kers} equal $1$. We now prove \eqref{eq:inner.prod.filter.sph.kers} for $T\geq1$. By \eqref{eq:filter.sph.ker} and \eqref{eq:innerL.zonal.kers},
\begin{align*}
 \InnerL{\vdh{T,\fil}(\PT{x}\cdot\cdot),\vdh{T,\fil}(\PT{z}\cdot\cdot)}
 &= \InnerL{\sum_{\ell=0}^{\infty}\fil\Bigl(\frac{\ell}{T}\Bigr)\:Z(d,\ell)\:\NGegen{\ell}(\PT{x}\cdot\cdot),
 \sum_{\ell'=0}^{\infty}\fil\Bigl(\frac{\ell'}{T}\Bigr)\:Z(d,\ell')\:\NGegen{\ell'}(\PT{z}\cdot\cdot)}\\[1mm]
 &= \sum_{\ell=0}^{\infty}\fil\Bigl(\frac{\ell}{T}\Bigr)^{2}\: Z(d,\ell)\:\NGegen{\ell}(\PT{x}\cdot\PT{z})
 = \vdh{T,\fil^2}(\PT{x}\cdot\PT{z}),
\end{align*}
thus completing the proof.
\end{proof}

When the filter is sufficiently smooth, the filtered kernel is strongly localized. This is shown in the following theorem proved by Narcowich et al. \cite[Theorem~3.5, p.~584]{NaPeWa2006-2}. For integer $\fis\geq0$, let $\CkR$ be the set of all $\fis$ times continuously differentiable functions on $\mathbb{R}_{+}$.
\begin{theorem}[\cite{NaPeWa2006-2}]\label{thm:filter.sph.ker.UB} Let $d\ge2$ and let $\fil$ be a filter in $\CkR$ with $1\leq \fis<\infty$ such that $\fil(t)$ is constant on $[0,a]$ for some $0<a<2$.
Then
\begin{equation}\label{eq:filter.sph.ker.UB}
    \bigl|\vdh{T,\fil}(\cos\theta)\bigr|\leq \frac{c \:T^{d}}{(1 + T \theta)^{\fis}}, \quad T\geq1,
\end{equation}
where $\cos\theta=\PT{x}\cdot\PT{y}$ for some $\PT{x},\PT{y}\in\sph{d}$ and the constant $c$ depends only on $d$, $\fil$ and $\fis$.
\end{theorem}
We give an alternative proof of Theorem~\ref{thm:filter.sph.ker.UB} in Section~\ref{sec:proofs.sec.filter.semidiscrete.needlet}, using different techniques.
\begin{remark} Dai and Xu \cite[Lemma~2.6.7, p.~48]{DaXu2013} proved \eqref{eq:filter.sph.ker.UB} for $\fil\in \CkR[3\fis+1]$. Brown and Dai \cite[Eq.~3.5, p.~409]{BrDa2005} and Narcowich et al. \cite[Theorem~2.2, p.~533]{NaPeWa2006-1} proved that for $\fil\in \CkR[\infty]$, \eqref{eq:filter.sph.ker.UB} holds for all positive integers $\fis$.
\end{remark}
The following theorem shows the boundedness of the $\mathbb{L}_{1}$-norm of the filtered kernel.
\begin{theorem}\label{thm:filter.sph.ker.L1norm.UB} Let $d\ge2$ and let $\fil$ be a filter in $\CkR$ with $\fis\ge \floor{\frac{d+3}{2}}$ such that $\fil(t)$ is constant on $[0,a]$ for some $0<a<2$. Then
\begin{equation}\label{eq:fiker.sph.norm.UB}
    \normb{\vdh{T,\fil}(\PT{x}\cdot\cdot)}{\Lp{1}{d}}\leq c_{d,\fil,\fis},\quad \PT{x}\in \sph{d},\;T\geq0.
\end{equation}
\end{theorem}
We give the proof of Theorem~\ref{thm:filter.sph.ker.L1norm.UB} in Section~\ref{sec:proofs.sec.filter.semidiscrete.needlet}.
\begin{remark}
    Narcowich et al. \cite{NaPeWa2006-2} proved \eqref{eq:fiker.sph.norm.UB} for $\fis\ge d+1$.
\end{remark}

Applying the convolution inequality of \cite[Eq.~1.14, p.~207--208]{BeBuPa1968} to \eqref{eq:filter.sph.approx} gives
\begin{equation*}
  \normb{\Vdh{T,\fil}(f)}{\Lp{p}{d}} \leq \normb{\vdh{T,\fil}(\PT{x}\cdot\cdot)}{\Lp{1}{d}} \norm{f}{\Lp{p}{d}}.
\end{equation*}
Thus by Theorem~\ref{thm:filter.sph.ker.L1norm.UB}, the operator norm of
the filtered approximation $\Vdh{T,\fil}$ on $\Lp{p}{d}$ is bounded for
$\fil$ satisfying the condition of Theorem~\ref{thm:filter.sph.ker.UB}:
\begin{corollary}\label{cor:filter.sph.operator.UB} Let $d\ge2$ and let $\fil$ be a filter in $\CkR$ with $\fis\ge \floor{\frac{d+3}{2}}$ such that $\fil(t)$ is constant on $[0,a]$ for some $0<a<2$ and let $1\leq p\leq \infty$. Then the filtered approximation $\Vdh{T,\fil}$ on $\Lp{p}{d}$ is an operator of strong type $(p,p)$,
i.e.
\begin{equation}\label{eq:fiapprox.norm.UB}
  \normb{\Vdh{T,\fil}}{\mathbb{L}_{p}\to \mathbb{L}_{p}}\leq c_{d,\fil,\fis},\quad T\geq0.
\end{equation}
\end{corollary}

For $L\in\mathbb{Z}_{+}$, the $\mathbb{L}_{p}$ \emph{error of best approximation} of order $L$ for $f\in \Lp{p}{d}$ is defined by
  $E_{L}(f)_{p} := E_{L}(f)_{\Lp{p}{d}}:=\inf_{p\in\sphpo{L}} \norm{f-p}{\Lp{p}{d}}$.

For given $f\in \Lp{1}{d}$ and $p\in[1,\infty]$, $E_{L}(f)_{p}$ is a non-increasing sequence. Since $\bigcup_{\ell=0}^{\infty}\sphpo{\ell}$ is dense in $\Lp{p}{d}$, the error of best approximation converges to zero as $L\to\infty$, i.e.
    $\lim_{L\to\infty}E_{L}(f)_{p}=0$, for $f\in\Lp{p}{d}$.

The error of best approximation for functions in a Sobolev space has the following upper bound, see \cite{Kamzolov1982} and also \cite[p.~1662]{MhNaPrWa2010}.
\begin{lemma}[\cite{Kamzolov1982,MhNaPrWa2010}]\label{lm:best.approx.sph.sob} Let $d\geq2$, $s>0$ and $1\leq p \leq \infty$. For $L\geq1$ and $f\in \sob{p}{s}{d}$,
\begin{equation*}
  E_{L}(f)_{p}\leq c\: L^{-s}\: \norm{f}{\sob{p}{s}{d}},
\end{equation*}
where the constant $c$ depends only on $d$, $p$ and $s$.
\end{lemma}

The filtered approximation $\Vdh{L,\fiH}$ has a near-best approximation error for sufficiently smooth $\fiH$ in the sense of being within a constant factor of a best approximation error, as shown by the following lemma.
\begin{theorem}\label{thm:fi.approx.err.via.best.approx} Let $d\geq2$ and $1\leq p \leq \infty$ and let $\fiH$ be the filter given by \eqref{eq:fiH} with $\fiN\in\CkR$ and $\fis\ge \floor{\frac{d+3}{2}}$. Then for $f\in \Lp{p}{d}$ and $L\geq1$,
    \begin{equation}\label{eq:filter.operator.via.best.approx}
  \normb{f-\Vdh{L,\fiH}(f)}{\Lp{p}{d}}\leq c \: E_{L}(f)_{p},
\end{equation}
where the constant $c$ depends only on $d$, $\fiH$ and $\fis$.
\end{theorem}
\begin{proof}
Since $\fiH(t)=1$ for $t\in [0,1]$, we have
\begin{equation*}
  \Vdh{L,\fiH}(p;\PT{x})=\InnerLb{p,\vdh{L,\fiH}(\PT{x}\cdot\cdot)} = p(\PT{x}),
\end{equation*}
and hence, from Corollary \ref{cor:filter.sph.operator.UB},
\begin{equation*} 
  \normb{f-\Vdh{L,\fiH}(f)}{\Lp{p}{d}}
  =\norm{f-p-\Vdh{L,\fiH}(f-p)}{\Lp{p}{d}}
   \leq \bigl(1+\normb{\Vdh{L,\fiH}}{\mathbb{L}_{p}\to \mathbb{L}_{p}}\bigr)\norm{f-p}{\Lp{p}{d}}
   \leq c_{d,\fiH} \: \|f-p\|_{\Lp{p}{d}},
\end{equation*}
which holds for all $p\in \sphpo{L}$, thus completing the proof.
\end{proof}
\begin{remark} The estimate \eqref{eq:filter.operator.via.best.approx} is a generalisation of the results of Rustamov \cite[Lemma~3.1, p.~316]{Rustamov1993} and Sloan \cite{Sloan2011}. Rustamov proved \eqref{eq:filter.operator.via.best.approx} for $\fiH\in \CkR[\infty]$ and $1\leq p\leq \infty$ while Sloan showed \eqref{eq:filter.operator.via.best.approx} for $p=\infty$, $f\in \ContiSph{}$ and $\fiH\in \CkR[d+1]$, and even for certain piecewise polynomial filters $\fiH$ belonging to $\CkR[d-1]$.
\end{remark}

Lemma~\ref{lm:best.approx.sph.sob} and Theorem~\ref{thm:fi.approx.err.via.best.approx} give the error of the filtered approximation for Sobolev spaces:
\begin{corollary}\label{cor:fi.approx.err.via.best.approx.sob} With the assumptions of Theorem~\ref{thm:fi.approx.err.via.best.approx}, for $f\in \sob{p}{s}{d}$ with $s>0$ and $L\geq1$,
\begin{equation*}
  \normb{f-\Vdh{L,\fiH}(f)}{\Lp{p}{d}}\leq c \: L^{-s} \:\norm{f}{\sob{p}{s}{d}},
\end{equation*}
where the constant $c$ depends only on $d$, $p$, $s$, $\fiH$ and $\fis$.
\end{corollary}
\begin{remark} H.~Wang and Sloan in a recent work \cite{WaSl2015} proved Corollary~\ref{cor:fi.approx.err.via.best.approx.sob} using essentially the same method.
\end{remark}

\subsection{Semidiscrete needlet approximation}\label{subsec:semidiscrete.needlet}
The smoothness of the filter makes the needlet $\needlet{jk}$ localized. This can be seen from the following corollary of Theorem~\ref{thm:filter.sph.ker.UB}, first proved by Narcowich et al. in \cite[Corollary~5.3, p.~592]{NaPeWa2006-2}.
\begin{corollary}[\cite{NaPeWa2006-2}]\label{cor:g.needlet.UB} Let $\fiN$ be a needlet filter, satisfying \eqref{subeqs:fiN}. If $\fiN\in \CkR$ with $\fis\ge d+1$, then
    \begin{equation*}
      |\needlet{jk}(\PT{x})|\leq \frac{c\: 2^{jd}}{\left(1 + 2^{j} \:\dist{\PT{x},\pN{jk}}\right)^{\fis}},\quad \PT{x}\in \sph{d},\; j\geq0,\; k=1,\dots,N_{j},
    \end{equation*}
where the constant $c$ depends only on $d$, $\fiN$ and $\fis$.
\end{corollary}

The following theorem shows, as foreshadowed in \eqref{eq:intro.needlets.vs.filter.sph.ker.fiN}, that an appropriate sum of products of needlets is exactly a filtered kernel. It is implicit in \cite{NaPeWa2006-1}.
\begin{theorem}[Needlets and filtered kernel]\label{thm:needlets.vs.filter.sph.ker} Let $\fiN$ be a needlet filter, see \eqref{subeqs:fiN}, and let $\fiH$ be given by \eqref{eq:fiH}. For $j\geq0$ and $1\leq k\leq N_{j}$, let $\needlet{jk}$ be needlets with filter $\fiN$ and needlet quadrature \eqref{eq:QN}. Then,
\begin{subequations}\label{subeqs:needlets.fi.ker}
\begin{align}
  \sum_{k=1}^{N_{j}}\needlet{jk}(\PT{x})\:\needlet{jk}(\PT{y})
  &=  \vdh{2^{j-1},\fiN^{2}}(\PT{x}\cdot\PT{y}),\;\; j\geq0,\label{eq:needlets.vs.filter.sph.ker.fiN}\\
    \sum_{j=0}^{\neord}\sum_{k=1}^{N_{j}}\needlet{jk}(\PT{x})\:\needlet{jk}(\PT{y})&=\vdh{2^{\neord-1},\fiH}(\PT{x}\cdot\PT{y}),\;\; \neord\geq0.
    \label{eq:needlets.vs.filter.sph.ker.fiH}
  \end{align}
\end{subequations}
\end{theorem}
For completeness we give a proof.
\begin{proof} For $j=0$, by \eqref{eq:needlets-a} and \eqref{eq:filter.sph.ker},
 \begin{equation*}
   \sum_{k=1}^{N_{0}}\needlet{0k}(\PT{x})\:\needlet{0k}(\PT{y})
   =\sum_{k=1}^{N_{0}}w_{0k}
   =\int_{\sph{d}}\IntDiff{z}
   =1=\vdh{2^{-1},\fiN^{2}}(\PT{x}\cdot\PT{y}).
 \end{equation*}
For $j\geq1$, using \eqref{eq:needlets-b} and the fact that the filter $\fiN$ has support $[1/2,2]$, we have (noting $\fiN(2)=0$)
 \begin{align}\label{eq:needlets.vs.filter.sph.ker.fiN-2}
\sum_{k=1}^{N_{j}}\needlet{jk}(\PT{x}) \:\needlet{jk}(\PT{y})
   &=\sum_{\ell=0}^{2^{j}-1} \sum_{\ell'=0}^{2^{j}-1} \fiN\Bigl(\frac{\ell}{2^{j-1}}\Bigr)\: \fiN\Bigl(\frac{\ell'}{2^{j-1}}\Bigr)\notag\\
   &\hspace{1cm} \times \sum_{k=1}^{N_{j}}\wN \:Z(d,\ell)\NGegen{\ell}(\PT{x}\cdot\pN{jk}) \:Z(d,\ell')\NGegen{\ell'}(\PT{y}\cdot\pN{jk}).
 \end{align}
Since $\{(\wN,\pN{jk}):k=1,\dots,N_{j}\}$ is exact for polynomials of degree $2^{j+1}-1$, the sum $\sum_{k=1}^{N_{j}}$ over quadrature points in \eqref{eq:needlets.vs.filter.sph.ker.fiN-2} is equal to the integral over $\sph{d}$. Then by \eqref{eq:orthogonal.NGegen} and the definition of the filtered kernel, see \eqref{eq:filter.sph.ker}, the equation \eqref{eq:needlets.vs.filter.sph.ker.fiN-2} gives
 \begin{equation}\label{eq:needlets.vs.filter.sph.ker.fiN-1}
\sum_{k=1}^{N_{j}}\needlet{jk}(\PT{x}) \:\needlet{jk}(\PT{y})
   = \sum_{\ell=0}^{\infty}\fiN\Bigl(\frac{\ell}{2^{j-1}}\Bigr)^{2} \:Z(d,\ell)\:\NGegen{\ell}(\PT{x}\cdot\PT{y})
   =\vdh{2^{j-1},\fiN^{2}}(\PT{x}\cdot\PT{y}).
 \end{equation}
This proves \eqref{eq:needlets.vs.filter.sph.ker.fiN}.

For $\neord\geq0$, by \eqref{eq:needlets.vs.filter.sph.ker.fiN-1} and \eqref{eq:fiH.fiN}, we now have, using $\fiN(0)=0$,
\begin{align*}
  \sum_{j=0}^{\neord}\sum_{k=1}^{N_{j}}\needlet{jk}(\PT{x}) \:\needlet{jk}(\PT{y})
  &= 1+\sum_{j=1}^{\neord}\sum_{\ell=1}^{\infty}\fiN\Bigl(\frac{\ell}{2^{j-1}}\Bigr)^{2} \:Z(d,\ell)\NGegen{\ell}(\PT{x}\cdot\PT{y})\\
  &= 1+\sum_{\ell=1}^{\infty}\left(\sum_{j=0}^{\neord-1}\fiN\Bigl(\frac{\ell}{2^{j}}\Bigr)^{2}\right)Z(d,\ell)\NGegen{\ell}(\PT{x}\cdot\PT{y})\\
  &= \sum_{\ell=0}^{\infty}\fiH\Bigl(\frac{\ell}{2^{\neord-1}}\Bigr) \:Z(d,\ell)\NGegen{\ell}(\PT{x}\cdot\PT{y}).
\end{align*}
This completes the proof.
\end{proof}

Theorem~\ref{thm:needlets.vs.filter.sph.ker} with
\eqref{eq:filter.sph.approx} leads to the following equivalence of the
filtered approximation with filter $\fiH$ and the semidiscrete needlet
approximation \eqref{eq:intro.needlet.approx}.
\begin{theorem}\label{thm:needlets.vs.filter.approx} Under the assumption of Theorem~\ref{thm:needlets.vs.filter.sph.ker}, for $f\in \Lp{1}{d}$ and $\neord\geq0$,
\begin{equation}\label{eq:needlets.vs.filter.approx}
    \Vdh{2^{\neord-1},\fiH}(f) = \sum_{j=0}^{\neord}\hspace{0.3mm}\sum_{k=1}^{N_{j}}\InnerL{f,\needlet{jk}}\needlet{jk}=\neapx(f).
\end{equation}
\end{theorem}

Theorems~\ref{thm:fi.approx.err.via.best.approx} and \ref{thm:needlets.vs.filter.approx} imply that the semidiscrete needlet approximation has a near-best approximation error.
\begin{theorem}\label{thm:needlets.err.L} Let $\neord\in\Zp$ and let $\neapx(f)$, see \eqref{eq:intro.needlet.approx}, be the semidiscrete needlet approximation with needlets $\needlet{jk}$, see \eqref{subeqs:needlets}, for filter smoothness $\fis\ge \floor{\frac{d+3}{2}}$. Then for $1\le p\le\infty$ and $f\in \Lp{p}{d}$,
\begin{equation*}
  \normb{f-\neapx(f)}{\Lp{p}{d}}\leq c \: E_{2^{\neord-1}}(f)_{p},
\end{equation*}
where the constant $c$ depends only on $d$, the filter $\fiN$ and $\fis$.
\end{theorem}

\begin{proof} By Theorem~\ref{thm:needlets.vs.filter.approx}, the approximation by the semidiscrete needlets $\neapx(f)$ is equivalent to that by filtered approximation $\Vdh{2^{\neord-1},H}(f)$. Then the definition \eqref{eq:intro.needlet.approx} of $\neapx(f)$ and \eqref{eq:filter.operator.via.best.approx} of Theorem~\ref{thm:fi.approx.err.via.best.approx}
together with Theorem~\ref{thm:needlets.vs.filter.approx} give
\begin{equation*}
   \normb{f-\neapx(f)}{\Lp{p}{d}}
   = \normb{f-\Vdh{2^{\neord-1},\fiH}(f)}{\Lp{p}{d}}
\leq c_{d,\fiH,\fis} \: E_{2^{\neord-1}}(f)_{p}.
\end{equation*}
\end{proof}

Theorem~\ref{thm:needlets.err.L} and Lemma~\ref{lm:best.approx.sph.sob} imply a rate of convergence of the approximation error of $\neapx(f)$ for $f$ in a Sobolev space, as follows.
\begin{theorem}\label{thm:needlets.err.W} Under the assumptions of Theorem~\ref{thm:needlets.err.L}, we have for $f\in \sob{p}{s}{d}$ with $s>0$ and $\neord\ge0$,
\begin{equation*}
  \normb{f-\neapx(f)}{\Lp{p}{d}}\leq c \: 2^{-\neord s} \norms{f},
\end{equation*}
where the constant $c$ depends only on $d$, $p$, $s$, $\fiN$ and $\fis$.
\end{theorem}

\section{Discrete needlet approximation}\label{sec:discrete.needlet.approx}

To implement the needlet approximation in a numerical computation, we need
to discretize the continuous inner product $\InnerL{f,\needlet{jk}}$ in
\eqref{eq:intro.needlet.approx}. We make use of the quadrature rule in
\eqref{eq:discrete.quadrature} to replace the continuous inner product
by a discrete version. In this section, we estimate the error by the
discrete needlet approximation for the Sobolev space $\sob{p}{s}{d}$,
$1\leq p \leq \infty$.

\subsection{Discrete needlets and filtered hyperinterpolation}
Let $\needlet{jk}$ be needlets satisfying \eqref{subeqs:needlets}, and let
$\QH:=\QH[](N,\ell):=\{(\wH,\pH{i}):i=1,\dots,N\}$ be a discretization
quadrature rule that is exact for polynomials of degree up to some $\ell$,
yet to be fixed. Applying the quadrature rule $\QH$ to the needlet
coefficient
  $\InnerL{f,\needlet{jk}} = \int_{\sph{d}} f(\PT{y})\needlet{jk}(\PT{y}) \IntDiff{y}$,
we obtain the discrete needlet coefficient
\begin{equation}\label{eq:f.needlet.discrete.inner.prod}
    \InnerD{f,\needlet{jk}} = \sum_{i=1}^{N} \wH\:f(\pH{i})\:\needlet{jk}(\pH{i}).
\end{equation}
This turns the semidiscrete needlet approximation \eqref{eq:intro.needlet.approx} into the (fully) discrete needlet approximation:
\begin{equation}\label{eq:discrete.needlet.approx}
    \disneapx(f) = \sum_{j=0}^{\neord}\sum_{k=1}^{N_{j}}\InnerD{f,\needlet{jk}} \needlet{jk}
\end{equation}

In a similar way to the semidiscrete case, cf. \eqref{eq:needlets.vs.filter.approx} of Theorem~\ref{thm:needlets.vs.filter.approx}, the discrete needlet approximation \eqref{eq:discrete.needlet.approx} is equivalent to filtered hyperinterpolation, which we now introduce.

The \emph{filtered hyperinterpolation approximation} with a filtered kernel $\vdh{T,\fil}$ in \eqref{eq:filter.sph.ker} and discretization quadrature $\QH$ in \eqref{eq:discrete.quadrature} is
\begin{align}\label{eq:filter.hyper}
    \DisVh{T,\fil,N}(f;\PT{x}) := \DisVh[d]{T,\fil,N}(f;\PT{x}) := \InnerDb{f,\vdh{T,\fil}(\cdot\cdot \PT{x})} := \sum_{i=1}^{N} \wH\: f(\pH{i})\: \vdh{T,\fil}(\pH{i}\cdot \PT{x}), \quad T\in\mathbb{R}_{+},
\end{align}
as named by Sloan and Womersley \cite{SlWo2012}; see also \cite{LeMh2008} and \cite{IvPe2014Fast}.

\begin{theorem}\label{thm:needlets.vs.filter.hyper} Let $\fiN$ be a needlet filter given by \eqref{subeqs:fiN} and let the filter $H$ be given by \eqref{eq:fiH}. For $f\in \ContiSph{}$ and $\neord\geq0$,
\begin{equation}\label{eq:needlets.vs.filter.hyper}
    \DisVh{2^{\neord-1},\fiH,N}(f) = \sum_{j=0}^{\neord}\sum_{k=1}^{N_{j}}\InnerD{f,\needlet{jk}} \needlet{jk} = \disneapx(f).
\end{equation}
\end{theorem}
\begin{remark}
    Note that in Theorem~\ref{thm:needlets.vs.filter.hyper} we do not yet require the number $N$ of nodes of the discretization quadrature to depend on
     the order $\neord$ of the discrete needlet approximation.
\end{remark}
\begin{proof}
Applying \eqref{eq:needlets.vs.filter.sph.ker.fiH} of
Theorem~\ref{thm:needlets.vs.filter.sph.ker} to
$\vdh{2^{\neord-1},\fiH}(\pH{i}\cdot \PT{x})$, cf.
\eqref{eq:filter.hyper}, and using
\eqref{eq:f.needlet.discrete.inner.prod}, we have
\begin{align*}
    \DisVh{2^{\neord-1},H,N}(f;\PT{x})& =  \sum_{i=1}^{N} \wH\: f(\pH{i}) \sum_{j=0}^{\neord}\sum_{k=1}^{N_{j}}\needlet{jk}(\pH{i})\needlet{jk}(\PT{x}) \notag\\
    & = \sum_{j=0}^{\neord}\sum_{k=1}^{N_{j}}\left(\sum_{i=1}^{N} \wH\: f(\pH{i})\:\needlet{jk}(\pH{i})\right)\needlet{jk}(\PT{x})
     = \sum_{j=0}^{\neord}\sum_{k=1}^{N_{j}}\InnerD{f,\needlet{jk}} \needlet{jk}(\PT{x}),
\end{align*}
which gives \eqref{eq:needlets.vs.filter.hyper}.
\end{proof}

\subsection{Error for filtered hyperinterpolation}

By Theorem~\ref{thm:needlets.vs.filter.hyper}, the discrete needlet approximation, if regarded as a function over the entire sphere, reduces to the filtered hyperinterpolation approximation. In this section, we deduce the approximation error of the filtered hyperinterpolation or discrete needlet approximation for $f$ in Sobolev spaces $\sob{p}{s}{d}$ with $1\leq p \leq \infty$ and $s>d/p$.

For comparison, by Corollary~\ref{cor:fi.approx.err.via.best.approx.sob}, the continuous approximation $\Vdh{L,\fiH}(f)$ has the following approximation error for $f\in\sob{p}{s}{d}$ with $1\leq p\leq \infty$ and $s>0$:
\begin{equation}\label{eq:filter.approx.err.W_2^s}
  \norm{f-\Vdh{L,\fiH}(f)}{\Lp{p}{d}} \leq c\: L^{-s}\: \norm{f}{\sob{p}{s}{d}}, \quad f\in \sob{p}{s}{d},
\end{equation}
where the constant $c$ depends only on $d$, $s$, filter $\fiH$ and $\fis$.

We now want error bounds for $\DisVh{L,\fiH,N}$.
For that discrete version of the filtered approximation, Mhaskar \cite[Theorems~3.1 and 3.2]{Mhaskar2006} obtained the same truncation error to that of \eqref{eq:filter.approx.err.W_2^s} for $f\in \sob{p}{s}{d}$ with $1\le p\le \infty$, see also \cite{LeMh2008,SlWo2012} which proved the case when $p=\infty$, and $s>d/p$, but with a stronger assumption on the filter smoothness than stated in Theorem~\ref{thm:filter.hyper.error.W^s.p} below. Given $\neord\in\Nz$, let $L:=2^{\neord-1}$ and let
\begin{equation}\label{eq:QH.3L}
    \QH:=\QH[](N,3L-1):=\left\{(\wH,\pH{i}): i=1,2,\dots,N\right\}
\end{equation}
be a discretization quadrature exact for polynomials of degree up to $3L-1$. The following theorem is due to H.~Wang and Sloan \cite{WaSl2015}.
\begin{theorem}\label{thm:filter.hyper.error.W^s.p} Let $d\ge2$, $1\le p\le \infty$ and $s> d/p$. Given a needlet filter $\fiN$, let $\DisVh{L,\fiH,N}$ be the filtered hyperinterpolation in \eqref{eq:filter.hyper} with $\QH$ given by \eqref{eq:QH.3L} and filter $\fiH$ given by \eqref{eq:fiH} and satisfying $H\in \CkR$ for $\fis\ge \floor{\frac{d+3}{2}}$. Then, for $f\in \sob{p}{s}{d}$,
    \begin{equation*}
    \normb{f-\DisVh{L,\fiH,N}(f)}{\Lp{p}{d}} \leq c\: L^{-s}\: \norm{f}{\sob{p}{s}{d}},
    \end{equation*}
    where the constant $c$ depends only on $d$, $p$, $s$, $\fiH$ and $\fis$.
\end{theorem}

Theorem~\ref{thm:filter.hyper.error.W^s.p} with Theorem~\ref{thm:needlets.vs.filter.hyper} gives the errors for the discrete needlet approximation of $f\in\sob{p}{s}{d}$, $1\leq p\leq
\infty$, as follows.
\begin{theorem}[Error by discrete needlets for $\sob{p}{s}{d}$]\label{thm:dis.needlets.err.Wp} Let $d\geq2$, $1\leq p\leq \infty$ and $s>d/p$, and let $\disneapx$ be the discrete needlet approximation given by \eqref{eq:intro.discrete.needlet.approx} with needlet filter $\fiN\in \CkR$ and $\fis\ge \floor{\frac{d+3}{2}}$ and with discretization quadrature $\QH$ in \eqref{eq:QH.3L}. Then, for $f\in \sob{p}{s}{d}$ and $\neord\in\Nz$,
    \begin{equation*}
    \normb{f-\disneapx(f)}{\Lp{p}{d}} \leq c\: 2^{-\neord s}\: \norm{f}{\sob{p}{s}{d}},
    \end{equation*}
    where the constant $c$ depends only on $d$, $p$, $s$, $\fiN$ and $\fis$.
\end{theorem}

\subsection{Discrete needlets and discrete wavelets}
Let $\QH$ be a discretization quadrature rule given by \eqref{eq:discrete.quadrature}. The discrete needlet approximation $\disneapx$ in \eqref{eq:intro.discrete.needlet.approx} can be written, for $f\in \ContiSph{}$ and $\PT{x}\in\sph{d}$, as
\begin{equation}\label{eq:VLN.by.Uj}
  \disneapx(f;\PT{x}) = \sum_{j=0}^{\neord}\disparsum{j,N}(f;\PT{x}),
\end{equation}
where $\disparsum{j,N}$ is the \emph{level-$j$ contribution} of the discrete needlet approximation defined by
\begin{equation}\label{eq:discrete.partial.sum.needlets}
  \disparsum{j,N}(f;\PT{x}):=\disparsum{j,N_{j}}(f;\PT{x}):=\sum_{k=1}^{N_{j}}\InnerD{f,\needlet{jk}}\needlet{jk}(\PT{x}),\quad f\in \ContiSph{},\: \PT{x}\in\sph{d}.
\end{equation}

Using \eqref{eq:needlets.vs.filter.sph.ker.fiN} then gives
\begin{equation}\label{eq:UjN.filtered.operator}
  \disparsum{j,N}(f;\PT{x})
  = \InnerDB{f,\sum_{k=1}^{N_{j}}\needlet{jk}(\cdot)\needlet{jk}(\PT{x})}
  = \InnerDB{f,\vdh{2^{j-1},\fiN^{2}}(\PT{x}\cdot\cdot)}
  = \DisVh{2^{j-1},\fiN^{2},N}(f;\PT{x}),
\end{equation}
where the filtered kernel $\vdh{2^{j-1},\fiN^{2}}(\PT{x}\cdot\PT{y})$ is given by \eqref{eq:filter.sph.ker}.

Using \eqref{eq:needlets.vs.filter.sph.ker.fiH} and \eqref{eq:UjN.filtered.operator} with \eqref{eq:filter.sph.approx} gives the following representation of filtered hyperinterpolation in terms of $\disparsum{j,N}$.
\begin{theorem}\label{thm:discrete.wavelets.vs.filter.hyper}
Let $d\geq2$ and let $\disparsum{j,N}(f)$ be the level-$j$ contribution of the discrete needlet approximation in \eqref{eq:VLN.by.Uj} and let $\fiH$ be the filter given by \eqref{eq:fiH}. Then for $f\in \ContiSph{}$ and $\neord\in\Nz$,
\begin{equation*}
    \DisVh{2^{\neord-1},\fiH,N}(f) = \sum_{j=0}^{\neord}\disparsum{j,N}(f).
\end{equation*}
\end{theorem}

Theorems~\ref{thm:needlets.vs.filter.hyper} and \ref{thm:discrete.wavelets.vs.filter.hyper} with \eqref{eq:filter.hyper} and \eqref{eq:UjN.filtered.operator} imply the following representation for $\disneapx$.
\begin{corollary} Let $\fiN$ be a needlet filter given by \eqref{subeqs:fiN} and let the filter $H$ be given by \eqref{eq:fiH}. For $f\in \ContiSph{}$ and $\neord\in\Nz$,
\begin{equation}\label{eq:disneapx.vdh}
  \disneapx(f;\PT{x}) = \sum_{j=0}^{\neord}\sum_{i=1}^{N_{j}} \wH\: f(\pH{i})\: \vdh{2^{j-1},\fiN^{2}}(\pH{i}\cdot \PT{x})
  = \InnerDb{f,\vdh{2^{\neord-1},\fiH}(\cdot\cdot\PT{x})}.
\end{equation}
\end{corollary}

The theorem below shows that the $\mathbb{L}_{p}$-norm of $\disparsum{j,N}(f)$ decays to zero exponentially with respect to order $j$. This means that the different levels of a discrete needlet approximation have different contributions and $\disparsum{j,N}(f)$ thus forms a multilevel decomposition. We can hence regard $\disparsum{j,N}(f)$ as a \emph{discrete wavelet transform}.
\begin{theorem}\label{thm:UB.contribution} Let $d\geq2$, $1\leq p\leq \infty$ and $s>d/p$ and let $\disparsum{j,N}$ be the level-$j$ contribution of the discrete needlet approximation in \eqref{eq:VLN.by.Uj} and let the needlet filter $\fiN$ satisfy $\fiN\in \CkR$ and $\kappa\ge \floor{\frac{d+3}{2}}$. Then for $f\in \sob{p}{s}{d}$ and $j\geq1$,
\begin{equation*}
    \normb{\disparsum{j,N}(f)}{\Lp{p}{d}} \leq c\: 2^{-j s}\:\norm{f}{\sob{p}{s}{d}},
\end{equation*}
where the constant $c$ depends only on $d$, $p$, $s$, $\fiN$ and $\fis$.
\end{theorem}
\begin{proof} Theorem~\ref{thm:discrete.wavelets.vs.filter.hyper} shows that $\disparsum{j,N}(f)$ is the difference of two filtered hyperinterpolation approximations:
for $j\geq1$,
\begin{equation*}
    \disparsum{j,N}(f) = \DisVh{2^{j-1},\fiH,N}(f) - \DisVh{2^{j-2},\fiH,N}(f).
\end{equation*}
This with Corollary~\ref{thm:filter.hyper.error.W^s.p} gives
\begin{align*}
    \normb{\disparsum{j,N}(f)}{\Lp{p}{d}}
    &\leq \normb{\DisVh{2^{j-1},\fiH,N}(f) - f}{\Lp{p}{d}} + \normb{f - \DisVh{2^{j-2},\fiH,N}(f)}{\Lp{p}{d}}\\
    &\leq c\: 2^{-j s}\:\norm{f}{\sob{p}{s}{d}},
\end{align*}
where the constant $c$ depends only on $d$, $p$, $s$, $\fiN$ and $\fis$.
\end{proof}

\section{Numerical examples}\label{sec:numerics}
In this section we give a computational strategy for discrete needlet approximation and show the results of some numerical experiments.
For the semidiscrete needlet case the approximation is not computable, but we are able to infer the error indirectly by using the Fourier-Laplace series of the test function to evaluate the $\mathbb{L}_{2}$ error.
The last part gives an example of a \emph{localized discrete needlet approximation} with high accuracy over a local region.
\subsection{Algorithm}

\begin{algorithm}\label{alg:discrete.needlet.approx}
Consider computing the discrete needlet approximation $\disneapx(f;\PT{x}'_{i})$ of order $\neord\in\Nz$ with needlet filter $\fiN$ at a set of points $\{\PT{x}'_{i}: i=1,\dots,{M}\}$. The needlet quadrature rules $\{(\wN,\pN{jk}):k=1,\dots,N_{j}\}$ are exact for polynomials of degree $2^{j+1}-1$.
 The major steps are analysis and synthesis.

\noindent 1. Analysis: Compute the discrete needlet coefficients $\InnerD{f, \needlet{jk}}$, $k=1,\dots,N_{j},\;j=0,\dots,\neord$ using a discretization quadrature rule $\QH=\QH[](N,3\cdot 2^{\neord-1}-1)=\{(\wH,\pH{i}):i=1,\dots,N\}$.\\
 2. Synthesis: Compute the discrete needlet approximation $\sum_{j=0}^{\neord}\sum_{k=1}^{N_{j}}\InnerD{f, \needlet{jk}} \needlet{jk}(\PT{x}'_{i})$, $i=1,\dots, M$.
\end{algorithm}

\noindent\textbf{Needlet filters.} Here $\needlet{jk}(\PT{x}'_{i})$ is computed by \eqref{eq:needlets-b} where the normalised Legendre polynomial $\NGegen{\ell}(t)$ is computed by the three-term recurrence formula, see \cite[\S~18.9(i)]{NIST:DLMF} and the needlet filter may be computed as follows. For construction of other needlet filters, see e.g. \cite{MaPi_etal2008,NaPeWa2006-1}.

Given $\fis\geq1$, let $p(t)$ be a polynomial of degree $2\fis+2$ of the form
\begin{equation}\label{eq:fiN.p}
  p(t) := \sum_{k=\fis+1}^{2\fis+2} a_{k} (1-t)^{k}, \quad t\in [0,1],
\end{equation}
where the coefficients $a_{k}$ are uniquely determined real numbers satisfying $p(0)=1$ and the $i$th derivatives of $p(t)$ at $t=0$ for $1\leq i\leq \kappa+1$ are zero. Clearly, $p(1)=0$ and all the $j$th derivatives of $p(t)$, $1\leq j\leq \kappa$, at $t=1$ are zero. Then it can be shown that
\begin{equation*}
  \fiN(t) := \begin{cases}
  p(t-1), & 1\leq t\leq 2,\\
  \sqrt{1-[p(2t-1)]^{2}}, & 1/2\leq t \leq 1,\\
  0, & \hbox{elsewhere}
  \end{cases}
\end{equation*}
is a filter $\fiN$ satisfying \eqref{subeqs:fiN}.
This section uses $\fis=5$, where the coefficients in \eqref{eq:fiN.p} are:
$a_{6}=924$, $a_{7}=-4752$, $a_{8}=10395$, $a_{9}=-12320$, $a_{10}=8316$, $a_{11}=-3024$, $a_{12}=462$, giving the filter $\fiN$ illustrated in Figure~\ref{fig:fiN}.

\begin{figure}
  \begin{minipage}{\textwidth}
  \centering
  \begin{minipage}{0.49\textwidth}
  \centering
  \includegraphics[trim = 1cm 5mm 1cm 4mm, width=0.85\textwidth]{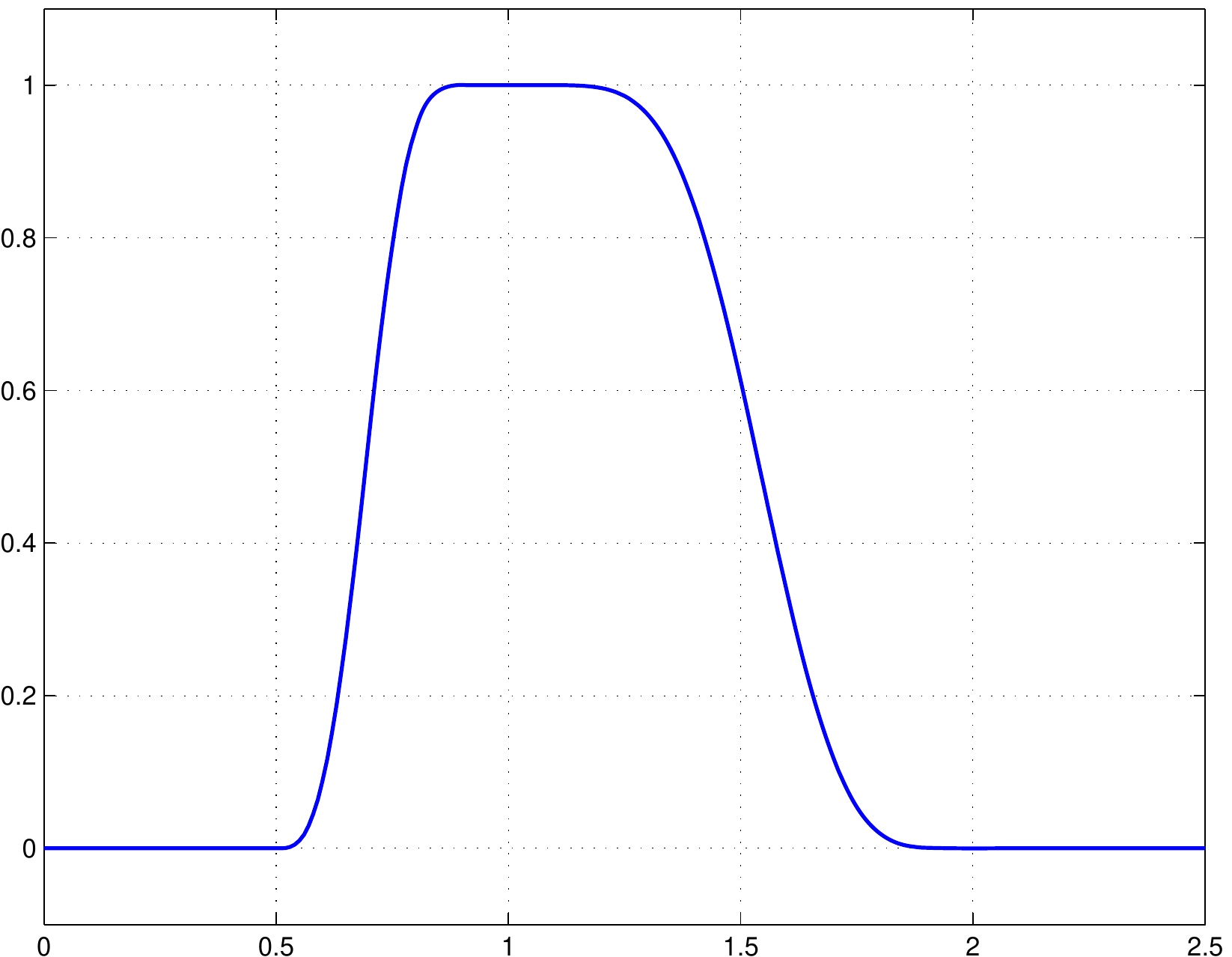}\\[2mm]
  \caption{Needlet filter $\fiN\in \CkR[5]$}\label{fig:fiN}
  \end{minipage}
  \begin{minipage}{0.49\textwidth}
  \centering
  \includegraphics[trim = 0cm -1cm 0cm 2mm, width=0.82\textwidth]{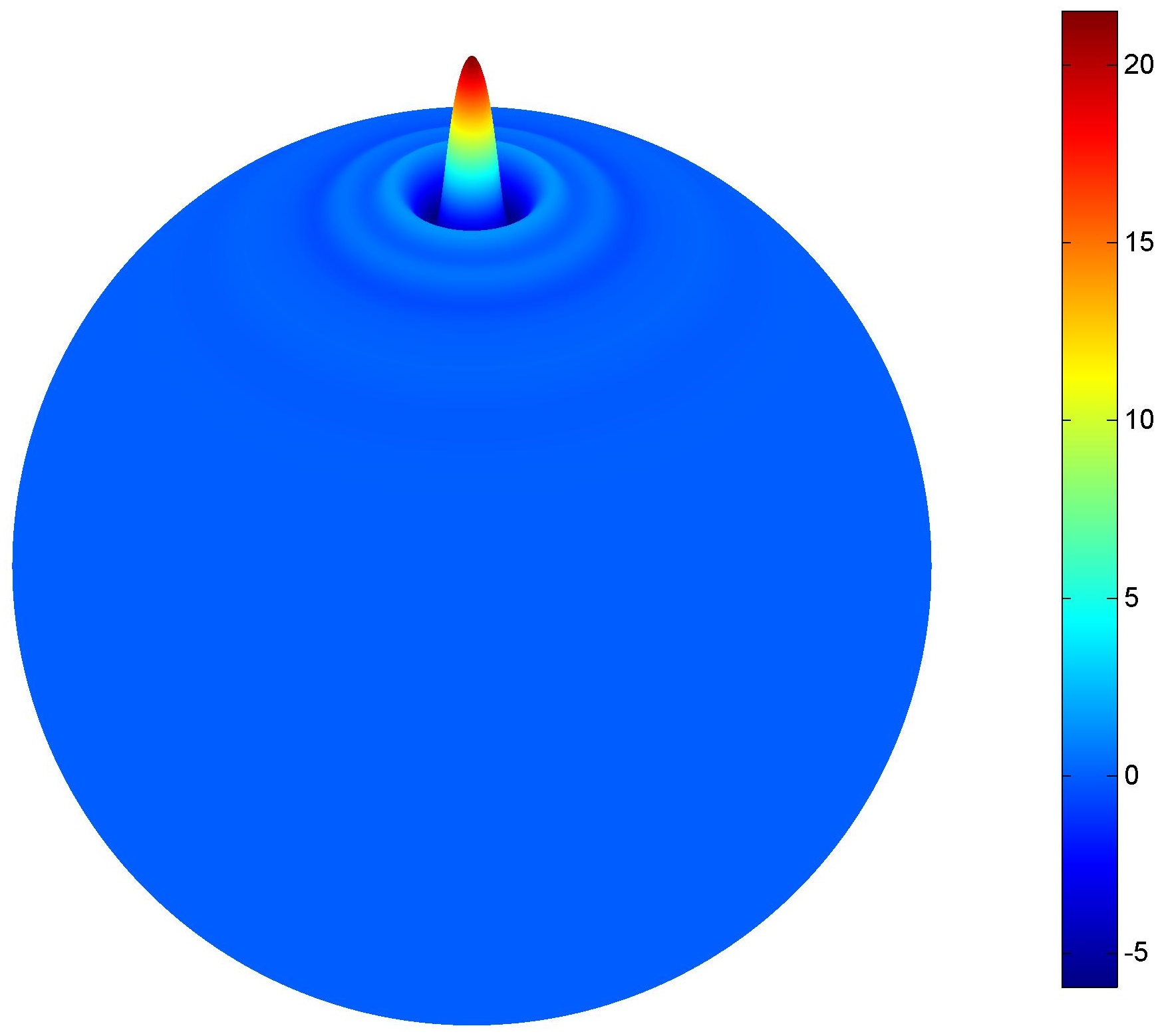}\\[-3mm]
  \caption{An order-$6$ needlet with a $C^{5}$-needlet filter}\label{fig:needlet.j6}
  \end{minipage}
  \end{minipage}
\end{figure}

Figure~\ref{fig:needlet.j6} shows an order-$6$ needlet with the filter given in Figure~\ref{fig:fiN}. We see that it is very localized.

\noindent\textbf{Quadrature rules.} We use \emph{symmetric spherical designs} for integration on $\sph{2}$, as recently developed by Womersley \cite{Womersley_ssd_URL}, for both the needlet quadrature rule and the discretization quadrature rule. Let $t$ be a non-negative integer. A symmetric (if $\PT{x}_{i}$ is a node so is $-\PT{x}_{i}$) spherical $t$-design is a quadrature rule with equal weights and exact for all polynomials of degree at most $t$. In these experiments the $t$-designs have $2\floor{\frac{t^2+t+4}{4}}\approx t^{2}/2$ nodes so are more efficient than $t$-designs using $(t+1)^{2}$ (which is the dimension of the polynomial space $\sphpo[2]{t}$) points, see for example \cite{AnChSlWo2012}. The symmetric spherical designs also have good geometric properties \cite{Womersley_ssd_URL}, which will benefit localized needlet approximations (see Section~\ref{eq:loc.approx.disneapx} below), compared with other quadrature rules such as \cite{HeWo2012} on $\sph{2}$.

\noindent\textbf{Cost of algorithm.}
Using a symmetric spherical $t$-design, a needlet quadrature rule for level $j$ has $N_{j}\approx 2^{2j+1}$ nodes, giving a total of
$\sum_{j=0}^{\neord}N_{j}\approx \frac{8}{3}\times 2^{2\neord}$ nodes for all $\neord$ levels as the symmetric spherical $t$-designs are not nested.
Similarly, a discretization quadrature rule exact up to degree $3\times 2^{\neord-1}-1$ has $N\approx \frac{9}{8}\times 2^{2\neord}$ nodes.
Thus the analysis step to evaluate the needlet coefficients requires $\frac{8}{3}\times 2^{2\neord}N$ evaluations of $f$.
The synthesis step only involves a weighted sum of the needlets evaluated at $M$ (possibly very large) points.
At high levels the number of needlets is large, for example when $\neord=6$, $L=64$, $N_{\neord}=8130$ and $N=4562$.

\subsection{Needlet approximation for the entire sphere}
This section illustrates the discrete needlet approximation of a function $f$ that is a linear combination of scaled Wendland radial basis functions on $\sph{2}$, see \cite{Wendland1995}. The advantage of this choice is that the Wendland functions have varying smoothness, and belong to known Sobolev spaces.

Let $(r)_{+}:=\max\{r,0\}$ for $r\in\mathbb{R}$. The original Wendland functions are \cite{,Wendland1995}
\begin{equation*}
  \fWend{k}(r) := \begin{cases}
  (1-r)_{+}^{2}, & k = 0,\\[1mm]
  (1-r)_{+}^{4}(4r + 1), & k = 1,\\[1mm]
  \displaystyle (1-r)_{+}^{6}(35r^2 + 18r + 3)/3, & k = 2,\\[1mm]
  (1-r)_{+}^{8}(32r^3 + 25r^2 + 8r + 1), & k = 3,\\[1mm]
  \displaystyle (1-r)_{+}^{10}(429r^4 + 450r^3 + 210r^2 + 50r + 5)/5, & k = 4.
  \end{cases}
\end{equation*}
The normalised (equal area) Wendland functions as defined in \cite{ChSlWo2014} are
\begin{equation*}
    \fnWend{k}(r) := \fWend{k}\Bigl(\frac{r}{\delta_{k}}\Bigr),\quad \delta_{k} := \frac{(3k+3)\Gamma(k+\frac{1}{2})}{2\:\Gamma(k+1)},\quad k\geq0.
\end{equation*}
The Wendland functions scaled this way have the property of converging pointwise to a Gaussian as $k\to \infty$, see Chernih et al. \cite{ChSlWo2014}. Thus as $k$ increases the main change is to the smoothness of $f$.
We write $\fnWend{}(r):=\fnWend{k}(r)$ for brevity if no confusion arises.

Let $\PT{z}_{1}:=(1,0,0)$, $\PT{z}_{2}:=(-1,0,0)$, $\PT{z}_{3}:=(0,1,0)$, $\PT{z}_{4}:=(0,-1,0)$, $\PT{z}_{5}:=(0,0,1)$, $\PT{z}_{6}:=(0,0,-1)$ be six points on $\sph{2}$ and define \cite{LeSlWe2010}
\begin{equation}\label{eq:Phi}
  f(\PT{x}) := f_{k}(\PT{x})
  := \sum_{i=1}^{6}\fnWend{k}(|\PT{z}_{i} - \PT{x}|), \quad k\geq0,
\end{equation}
where $|\cdot|$ is the Euclidean distance.

Narcowich and Ward \cite{NaWa2002} and Le Gia, Sloan and Wendland \cite{LeSlWe2010} proved that $f_{k}\in \sobH[2]{k+\frac{3}{2}}$.
Figure~\ref{fig:RBF.k2.xc} shows the picture of $f_{2}$, which belongs to $\sobH[2]{\frac{7}{2}}$. The function $f_{k}$ has limited smoothness at the centers $\PT{z}_{i}$ and at the boundary of each cap with center $\PT{z}_{i}$. These features make $f_{k}$ relatively difficult to approximate in these regions, especially for small $k$.
\begin{figure}
  \centering
  \begin{minipage}{0.5\textwidth}
  \centering
  \includegraphics[trim = 0mm 0mm 0mm 1mm, width=\textwidth]{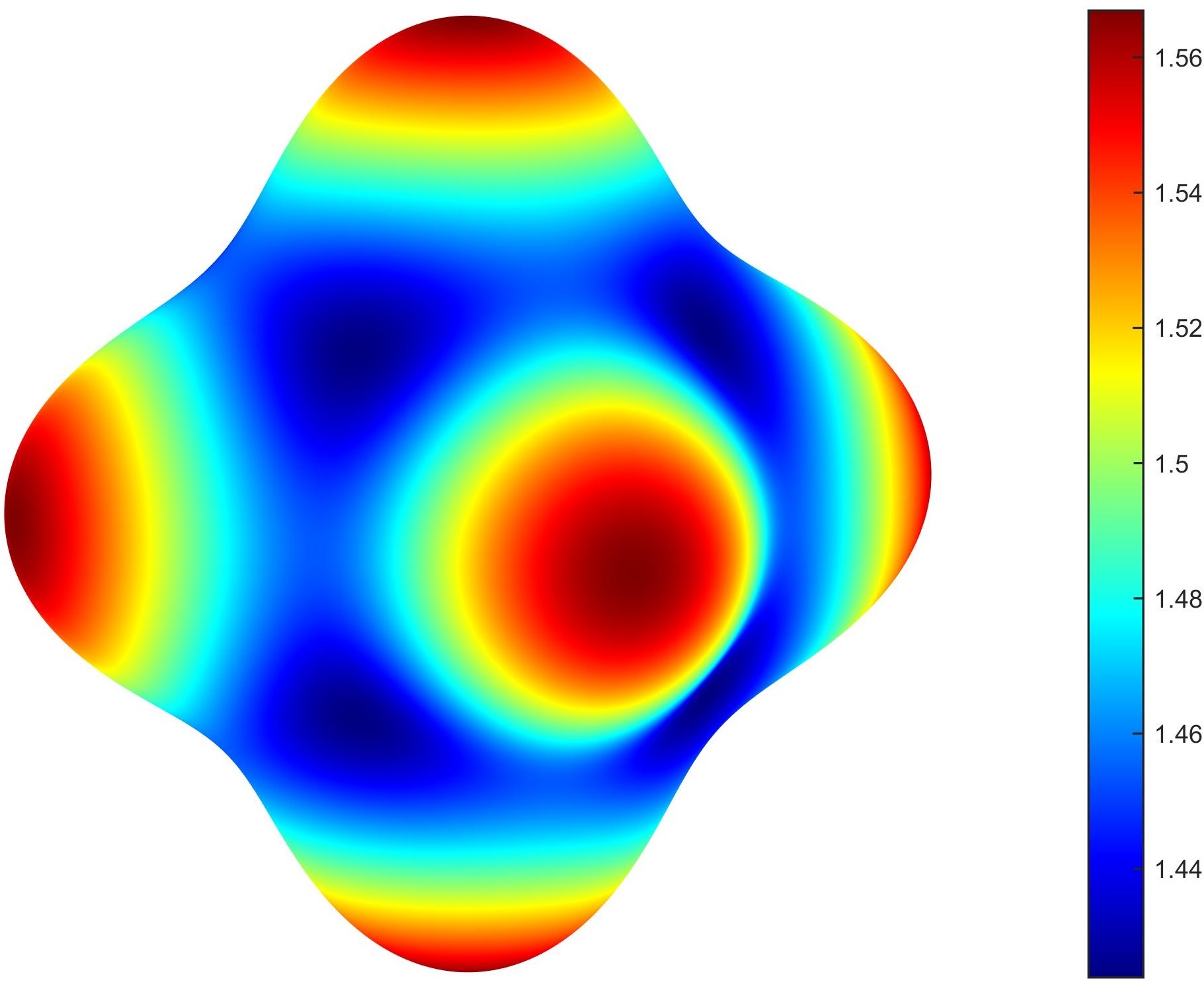}\\[1mm]
  \caption{The test function $f_{2}$}\label{fig:RBF.k2.xc}
  \end{minipage}
\end{figure}

\noindent\textbf{$\mathbb{L}_{2}$ approximation error.} We show the $\mathbb{L}_{2}$ errors when using $\neapx$ and by $\disneapx$. For $\disneapx(f)$ we compute its $\mathbb{L}_{2}$ error by discretizing the squared $\mathbb{L}_{2}$-norm by a quadrature rule. We cannot compute the $\mathbb{L}_{2}$ error for $\neapx(f)$ in this way as we do not have access to exact integrals for the inner products. As the test function in \eqref{eq:Phi} is a linear combination of Wendland functions, we are able to approximate the $\mathbb{L}_{2}$ error of $\neapx(f)$ by truncating the Fourier-Laplace expansion and using the known Fourier coefficients of Wendland functions.

We make use of the Fourier-Laplace coefficients of $f$ to compute the $\mathbb{L}_{2}$-error of the semidiscrete needlet approximation over the entire sphere, as follows. By Theorem~\ref{thm:needlets.vs.filter.approx} and the definition of the filtered approximation, see \eqref{eq:filter.sph.ker} and \eqref{eq:filter.sph.approx}, and the addition theorem, see \eqref{eq:addition.theorem}, the Fourier coefficients of $\neapx(f)$ are $\fiH(\ell/2^{\neord-1})\:\Fcoe{f}$. Then the Parseval's identity gives
\begin{equation}\label{eq:L2.err.VL.Phi}
  \normb{f - \neapx(f)}{\Lp{2}{2}}^{2}
  = \sum_{\ell = 2^{\neord-1}+1}^{\infty} \sum_{m=1}^{2\ell+1}\left(1-\fiH\left(\tfrac{\ell}{2^{\neord-1}}\right)\right)^{2}|\Fcoe{f}|^{2}.
\end{equation}

We expand $\phi(\sqrt{2-2t})$ in terms of $\Legen{\ell}(t)$:
\begin{equation*}
    \fnWend{}(\sqrt{2-2t}) = \sum_{\ell=0}^{\infty} \Fcoe[\ell]{\fnWend{}}\: (2\ell+1) \Legen{\ell}(t),
\end{equation*}
where $\Legen{\ell}(t)$ is the Legendre polynomial of degree $\ell$ and
\begin{equation}\label{eq:Fcoe.phi}
  \Fcoe[\ell]{\fnWend{}} :=  \frac{1}{2}\int_{-1}^{1}\fnWend{}(\sqrt{2-2t})\Legen{\ell}(t)\IntD{t},\quad \ell\geq0.
\end{equation}
Using the addition theorem again,
\begin{align*}
    \fnWend{}(|\PT{z}_{i} - \PT{x}|) = \fnWend{}\left(\sqrt{2 - 2 \:\PT{z}_{i}\cdot\PT{x}}\right) = \sum_{\ell=0}^{\infty} \Fcoe[\ell]{\fnWend{}}\: (2\ell+1) \Legen{\ell}(\PT{z}_{i}\cdot\PT{x})
    = \sum_{\ell=0}^{\infty} \sum_{m=1}^{2\ell+1} \Fcoe[\ell]{\fnWend{}}\: Y_{\ell,m}(\PT{z}_{i}) Y_{\ell,m}(\PT{x}),
\end{align*}
which with \eqref{eq:Phi} gives
\begin{equation*}
  \Fcoe{f} = \InnerL{f, Y_{\ell,m}} = \Fcoe[\ell]{\fnWend{}} \sum_{i=1}^{6} Y_{\ell,m}(\PT{z}_{i}).
\end{equation*}
This with \eqref{eq:L2.err.VL.Phi} and the addition theorem together gives
\begin{align}\label{eq:L2.err.VL.Phi-b}
  \normb{f - \neapx(f)}{\Lp{2}{d}}^{2}
  &= \sum_{\ell = 2^{\neord-1}+1}^{\infty} \sum_{m=1}^{2\ell+1}\left(1-\fiH\left(\tfrac{\ell}{2^{\neord-1}}\right)\right)^{2}\bigl|\Fcoe[\ell]{\fnWend{}}\bigr|^{2} \left(\sum_{i=1}^{6} Y_{\ell,m}(\PT{z}_{i})\right)^{2}\notag\\
  &= \sum_{\ell = 2^{\neord-1}+1}^{\infty} \left(1-\fiH\left(\tfrac{\ell}{2^{\neord-1}}\right)\right)^{2} \bigl|\Fcoe[\ell]{\fnWend{}}\bigr|^{2}\sum_{i=1}^{6}\sum_{j=1}^{6} (2\ell+1)\Legen{\ell}(\PT{z}_{i}\cdot\PT{z}_{j}),
\end{align}
where we use the Gauss-Legendre rule to compute the one-dimensional integral \eqref{eq:Fcoe.phi} for $\Fcoe[\ell]{\phi}$ to the desired accuracy.

Figure~\ref{fig:L2err.semineapx.RBFs} shows the $\mathbb{L}_{2}$-error of the semidiscrete needlet approximation $\neapx(f_{k})$ for $k=0,1,2,3,4$, where we used the filter $\fiN$ of Figure~\ref{fig:fiN}, with $\fiH$ then given by \eqref{eq:fiH.fiN}, and the order of semidiscrete needlet approximation is $\neord=1,\dots,6$, and the truncation degree $\ell$ in \eqref{eq:L2.err.VL.Phi-b} is taken as high as $500$. The slight fluctuation of the $\mathbb{L}_{2}$-errors of the semidiscrete needlet approximation for $f_{4}$ is partly due to the truncation error for the Fourier coefficients of $\fnWend{4}$. We use a log-log plot with the horizontal axis indexed by orders $\neord$ and fit data from $\neord=3$ to $6$ for each $k$ to illustrate the convergence order.
\begin{figure}
  \centering
  \begin{minipage}{\textwidth}
  \begin{minipage}{0.48\textwidth}
  \centering
  \includegraphics[trim = 0mm 3mm 0mm 1mm, width=1\textwidth]{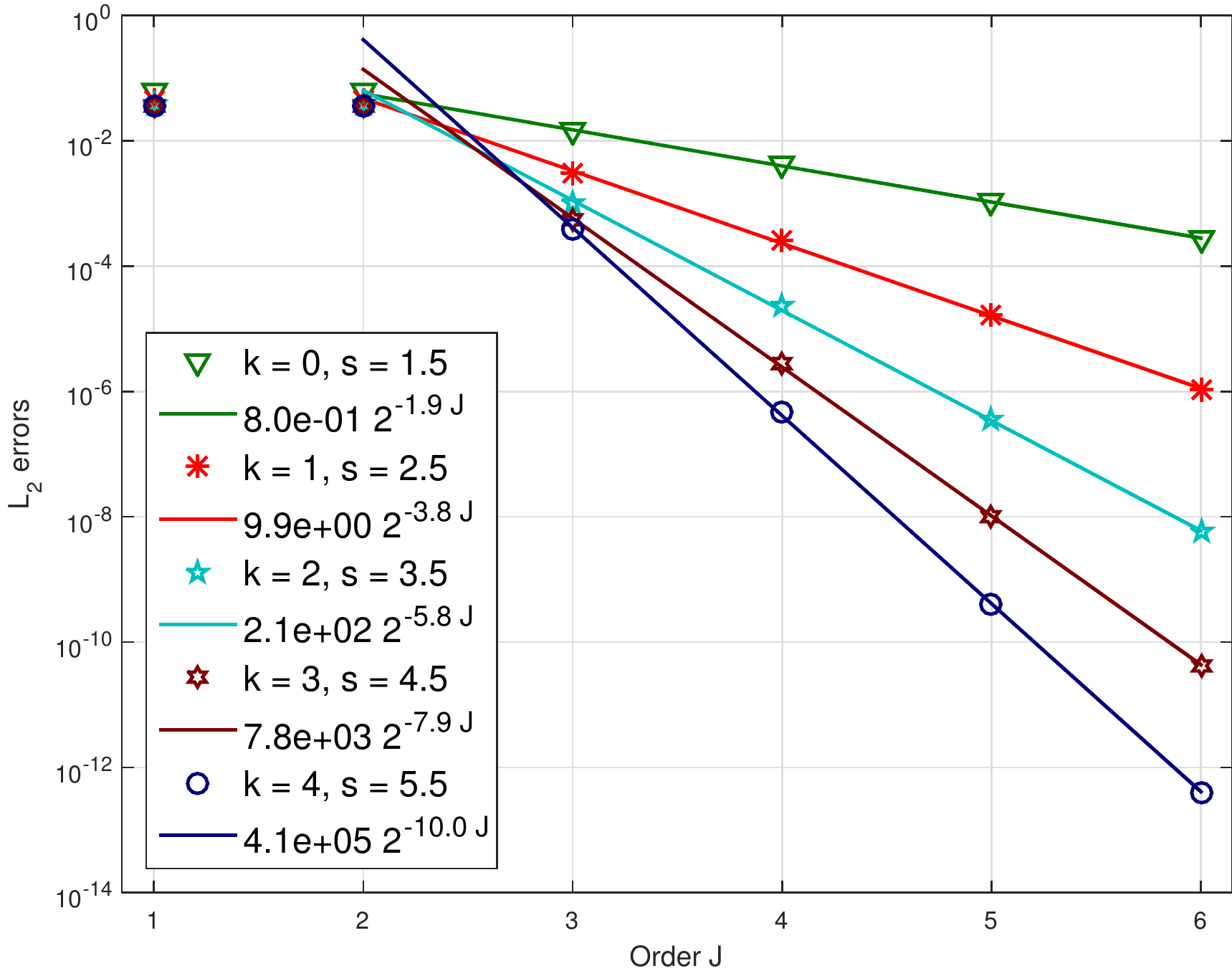}\\[1mm]
  \subcaption{Semidiscrete}\label{fig:L2err.semineapx.RBFs}
  \end{minipage}
  \hspace{0.01\textwidth}
  \begin{minipage}{0.48\textwidth}
  \centering
  \includegraphics[trim = 0mm 3mm 0mm 1mm, width=0.97\textwidth]{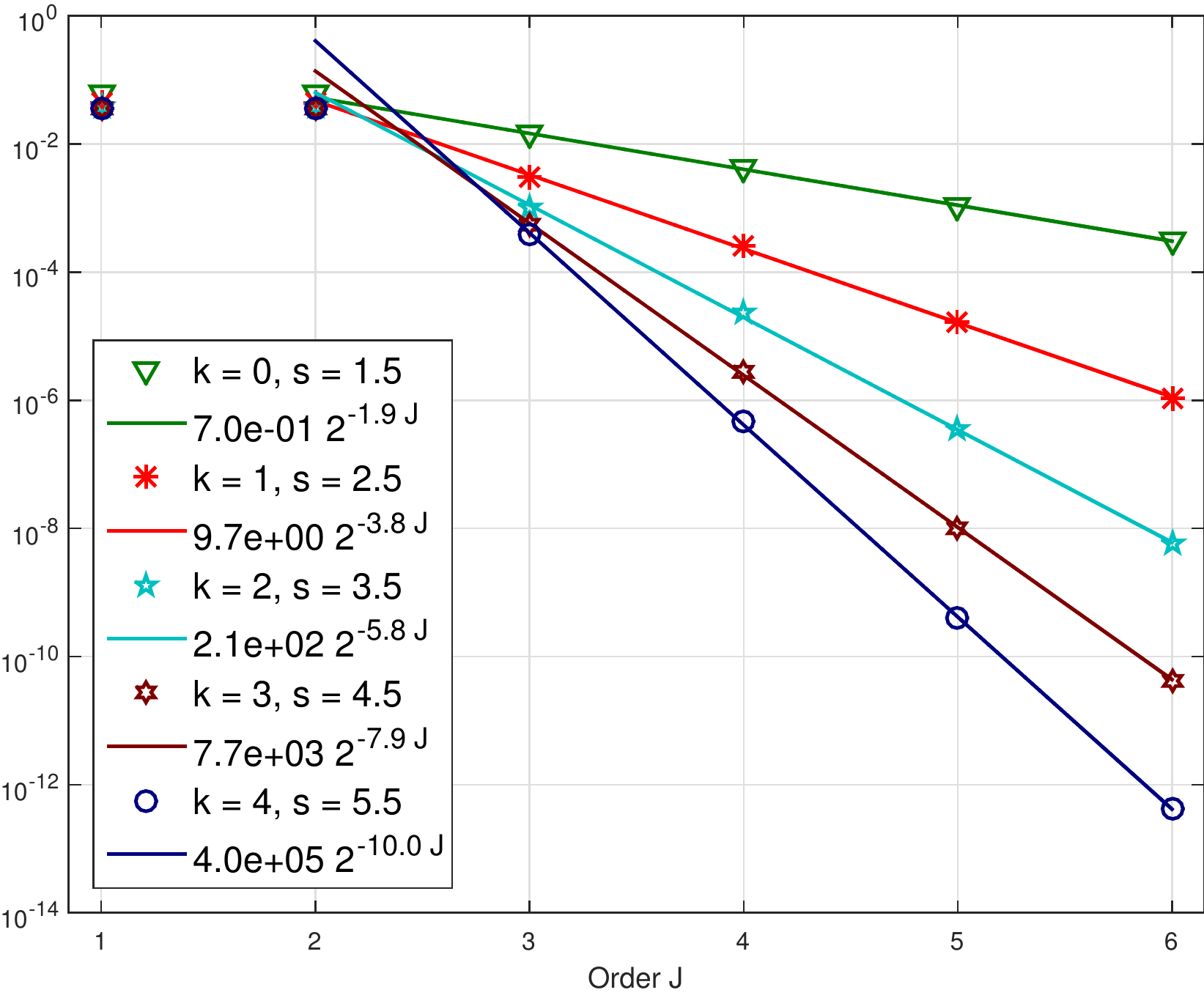}\\[1mm]
  \subcaption{Fully discrete}\label{fig:L2err.disneapx.RBFs}
  \end{minipage}\\[-1mm]
  \caption{$\mathbb{L}_{2}$-errors for needlet approximations of $f_{k}$, with orders $\neord=1,\dots,6$, filter $\fiN\in \CkR[5]$}\label{fig:L2err.neapx.RBFs}
  \end{minipage}
\end{figure}

Either \eqref{eq:disneapx.vdh} and \eqref{eq:filter.sph.ker}, or the needlet decomposition \eqref{eq:intro.discrete.needlet.approx} can be used to compute the fully discrete needlet approximation $\disneapx(f)$. Some discussion of efficient implementation can be found in \cite{IvPe2014Fast}. We then approximate the $\mathbb{L}_2$ error by a quadrature rule $\bigl\{(\widetilde{w}_{i},\PT{x}_{i}):i=1,\dots,\widetilde{N}\bigr\}$, as follows.
\begin{equation}\label{eq:discrete.L2.err.disneapx}
    \normb{\disneapx(f) - f}{\Lp{2}{2}}^{2}
    = \int_{\sph{2}} \bigl|\disneapx(f;\PT{x}) - f(\PT{x})\bigr|^{2} \IntDiff{x}
    \approx \sum_{i=1}^{\widetilde{N}} \widetilde{w}_{i}\: \bigl(\disneapx(f;\PT{x}_{i}) - f(\PT{x}_{i})\bigr)^{2}.
\end{equation}

Figure~\ref{fig:L2err.disneapx.RBFs} shows the corresponding $\mathbb{L}_{2}$-error for the discrete needlet approximation $\disneapx(f_{k})$, where we used the same needlet filter, and used symmetric spherical designs for both needlets and discretization, and the order of discrete needlet approximation is $\neord=1,\dots,6$. We used a symmetric spherical $301$-design (with $\widetilde{N}=45454$ nodes and equal weights $\widetilde{w}_{i}=1/\widetilde{N}$) to approximate the integral in \eqref{eq:discrete.L2.err.disneapx}.

For each $k$, the $\mathbb{L}_{2}$-errors of the semidiscrete and fully discrete needlet approximations converge at almost the same order (with respect to $2^{\neord}$). For the particular test function $f_{k}$ the order of convergence is close to $2^{-(2s-1)\neord}$, better than the optimal order when considering all functions in $\sobH[2]{s}$. The figure also shows that the convergence order becomes higher as the smoothness of $f$ increases, which is consistent with the theory.

\subsection{Local approximation by discrete needlets}\label{eq:loc.approx.disneapx}
In the following example, we show the approximation error using discrete needlets for $f_{2}$ given by \eqref{eq:Phi}, using all needlets at low levels and needlets with centers in a small region at high levels.

In general, let $X$ be a compact set of $\sph{d}$. We define the \emph{localized discrete needlet approximation} for $f\in \ContiSph{}$ by
\begin{equation*}
  \trneapx[\trneord,\neord,N](X;f;\PT{x}) :=
  \begin{cases}
  \displaystyle \sum_{j=0}^{\trneord}\disparsum{j,N}(f;\PT{x}),\quad \PT{x}\in \sph{d}\backslash X,\\[5mm]
  \displaystyle \sum_{j=0}^{\trneord}\disparsum{j,N}(f;\PT{x}) + \hspace{-3mm}\sum_{\PT{x}_{jk}\in X\atop \trneord+1\leq j\leq\neord}\InnerD{f,\needlet{jk}}\needlet{jk}(\PT{x}),\quad \PT{x}\in X,
  \end{cases}
\end{equation*}
where $\disparsum{j,N}(f;\PT{x})$, given by \eqref{eq:discrete.partial.sum.needlets}, is the level-$j$ contribution of the discrete needlet approximation. The idea is that on the compact set $X$ we seek a more refined needlet approximation --- that is, we ``zoom-in'' on the set $X$.

Let $X:=\scap{\PT{z}_{3},r}$, the spherical cap with center $\PT{z}_{3}:=(0,1,0)$ and radius $r$.
\begin{figure}
 \centering
  \begin{minipage}{\textwidth}
  \centering
\begin{minipage}{0.485\textwidth}
\centering
  \includegraphics[trim = 0mm -20mm 10mm 20mm, width=0.85\textwidth]{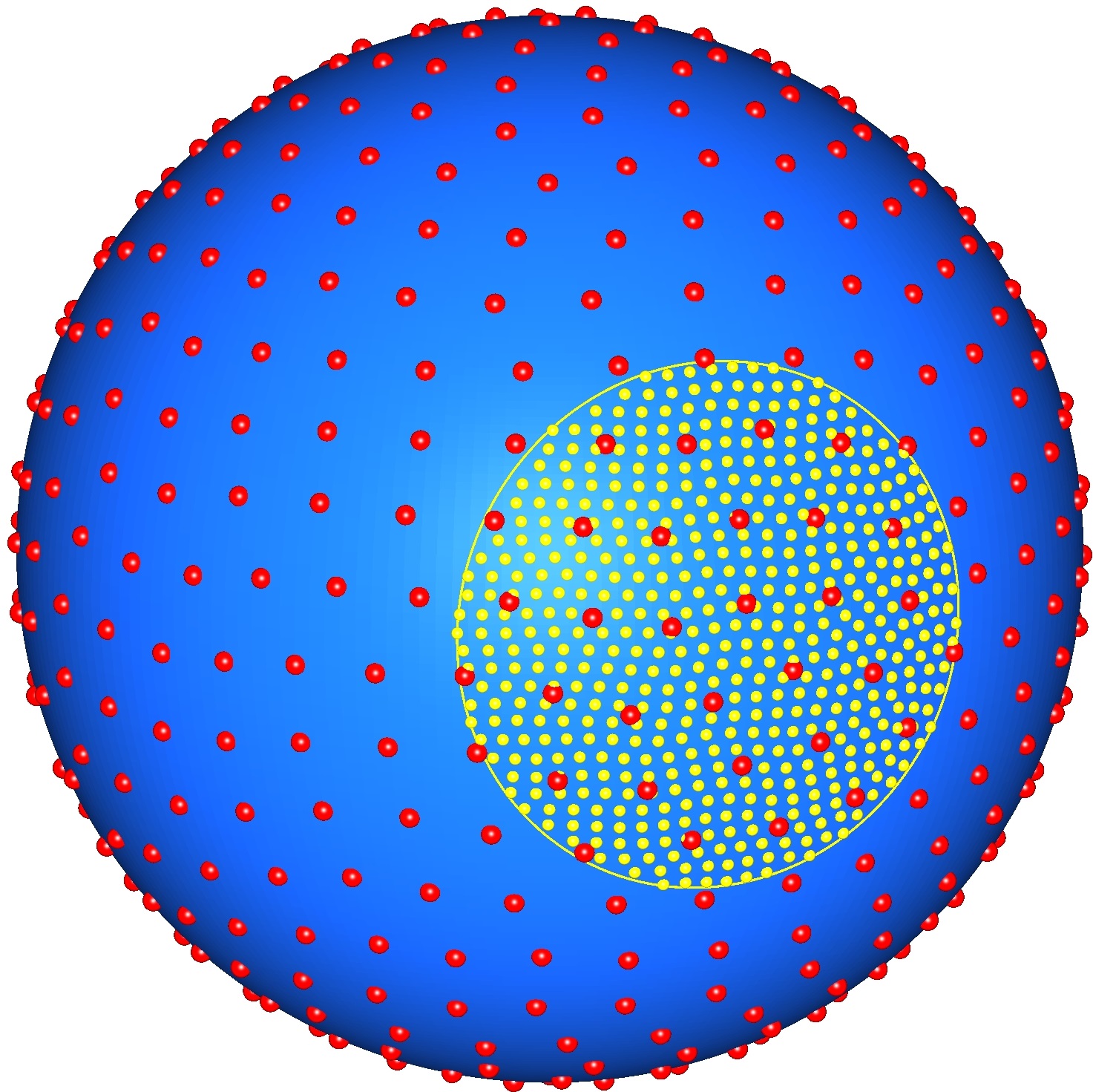}\\[-2mm]
  \begin{minipage}{0.88\textwidth}
  \caption{Centers of needlets at level $4$ (larger points) and level $6$ (smaller points) for a localized discrete needlet approximation}\label{fig:points.Jt4.J6.6thpi}
  \end{minipage}
\end{minipage}
\begin{minipage}{0.485\textwidth}
\centering
  \includegraphics[trim = 0mm 2mm 0mm 0mm, width=1.05\textwidth]{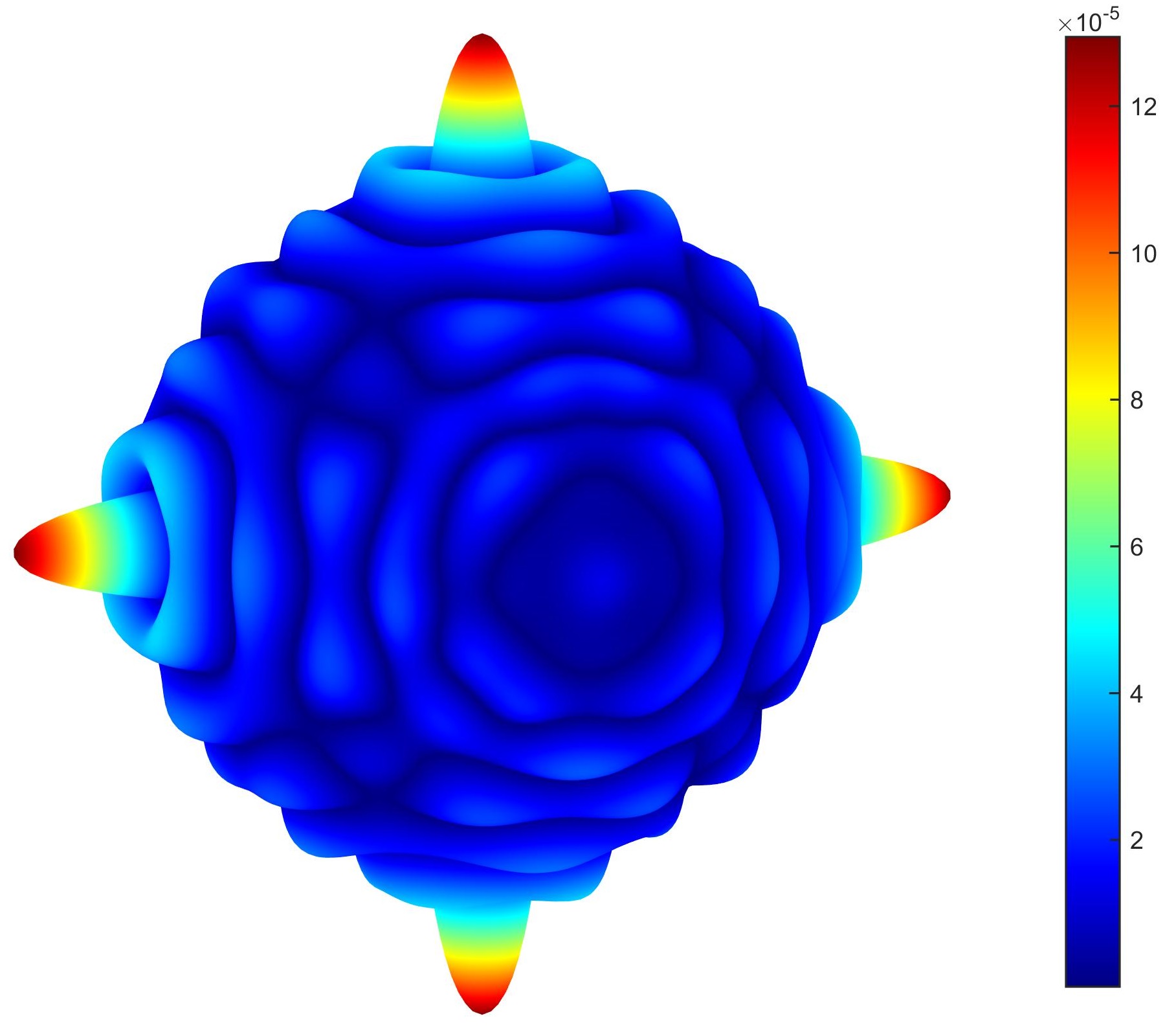}\\
  \begin{minipage}{0.88\textwidth}
  \caption{Absolute errors of a localized discrete needlet approximation for $f_{2}$ with full needlets at levels $\leq 4$ and localized needlets at levels $5,6$}\label{fig:RBF.k2.xc.Jt4.J6.6thpi}
  \end{minipage}
\end{minipage}
\begin{minipage}{0.03\textwidth}
\end{minipage}
\end{minipage}\vspace{-2mm}
\end{figure}

Figure~\ref{fig:RBF.k2.xc.Jt4.J6.6thpi} shows the pointwise absolute error of the localized discrete needlet approximation\\ $\trneapx[4,6,N](\scap{\PT{z}_{3},\pi/6};f_{2};\PT{x})$. In the exterior of the cap $\scap{\PT{z}_{3},\pi/6}$, the approximation used needlets up to level $4$, with the largest absolute error, about $1.2\times 10^{-4}$, at the centers $\PT{z}_{i}$, $i\ne3$. In the cap, the approximation is a combination of needlets at the low levels $0$ to $4$ with those at high levels $5$ and $6$.

We observe that the localized discrete needlet approximation has good approximation near the center of the local region but with less computational cost since the levels $5$ and $6$ used only a fraction of the full set of needlets, approximately $|\scap{\PT{z}_{3},r}|/|\sph{2}|=(1-\cos(r))/2$ (about $6.7\%$ when $r=\pi/6$). This localization is an efficient way of constructing a discrete needlet approximation for a specific region.

Figure~\ref{fig:points.Jt4.J6.6thpi} shows the centers of the needlets for level $4$ (larger points) and those in the cap $\scap{\PT{z}_{3},\pi/6}$ for level $6$ (smaller points) of the localized discrete needlet approximation $\trneapx[4,6,N](\scap{\PT{z}_{3},\pi/6};f_{2};\PT{x})$. The smaller points illustrate where the high levels of the localized discrete needlet approximation focused. At level $6$, the needlet quadrature used the symmetric spherical $63$-design, which has totally $8130$ nodes over the sphere and $544$ nodes in the cap. The localized discrete needlet approximation at this level used only needlets with centers at these $544$ nodes for the local region.

\section{Proofs for Section~\ref{sec:filter.semidiscrete.needlet}}\label{sec:proofs.sec.filter.semidiscrete.needlet}

In the proof of Theorem \ref{thm:filter.sph.ker.UB} we use the following
lemma to bound $\Ak{k}(T,\ell)$, where

\begin{equation}\label{eq:def. A_k(n,l)}
  \Ak{k}(T,\ell) :=\left\{\begin{array}{ll}
  \displaystyle \fil\Bigl(\frac{\ell}{T}\Bigr) -\fil\Bigl(\frac{\ell+1}{T}\Bigr), & k=1,\\[0.3cm]
  \displaystyle\frac{\Ak{k-1}(T,\ell)}{2\ell+2r+k}-\frac{\Ak{k-1}(T,\ell+1)}{2(\ell+1)+2r+k}, & k=2,3,\dots.
  \end{array}\right.
\end{equation}
Note that $\Ak{k}(T,\ell)$ vanishes for $\ell\le \lceil aT\rceil-k$,
because of the assumed constancy of $\fil$ on $[0,a]$. This is a crucial
property for establishing the following lemma.

\begin{lemma}\label{lm:est. Ak} Let $\fil$ satisfy the condition of Theorem~\ref{thm:filter.sph.ker.UB} with $T_{1}=\lceil aT\rceil$
and $T_{2}=\lfloor 2T\rfloor$ and with $T$ sufficiently large that $0\leq T_{1}-\fis\leq T_{2}$.
Let $\Ak{k}(T,\ell)$ be defined by \eqref{eq:def. A_k(n,l)}.  Then for an
arbitrary  positive integer $k\leq \fis$,
\begin{equation}\label{eq:est.Ak}
  \Ak{k}(T,\ell)=\bigo{}{T^{-(2k-1)}}, \quad T_{1}-k\leq \ell\leq T_{2},
\end{equation}
where the constant in the big $\mathcal{O}$ depends only on $d$, $k$,
$\fil$ and $\fis$.
\end{lemma}

\begin{proof} For a sequence $u_{\ell}$, let $\FDiff{1}\: u_{\ell}:=\FDiff{1} (u_{\ell}):=u_{\ell}-u_{\ell+1}$ denote the first order forward difference of $u_{\ell}$, and for $i\geq2$, let the $i$th order forward difference be defined recursively by $\FDiff{i} (u_{\ell}):=\FDiff{1}\bigl(\FDiff{i-1}(u_{\ell})\bigr)$. We now prove the estimate in \eqref{eq:est.Ak}, making use of the obvious identity
\begin{equation}\label{eq:Delta.uell.nuell}
    \FDiff{1}\: (u_{\ell}\: \nu_{\ell} )= (\FDiff{1}\: u_{\ell})\: \nu_{\ell} + u_{\ell+1}\:(\FDiff{1}\: \nu_{\ell}).
\end{equation}

By \eqref{eq:def. A_k(n,l)}, for $k\geq2$
\begin{align*}
    \Ak{k}(T,\ell)
    &= \left(\frac{\Ak{k-1}(T,\ell)}{2\ell+2r+k}-\frac{\Ak{k-1}(T,\ell)}{2(\ell+1)+2r+k}\right)
    +\left(\frac{\Ak{k-1}(T,\ell)}{2(\ell+1)+2r+k}-\frac{\Ak{k-1}(T,\ell+1)}{2(\ell+1)+2r+k}\right)\\[1mm]
    &= \frac{1}{2\ell+2r+k+2}\left(\frac{2}{2\ell+2r+k}+\FDiff{1}\right) \Ak{k-1}(T,\ell)
    =: \delta_{k}(\ell) \bigl(\Ak{k-1}(T,\ell)\bigr).
\end{align*}
In addition, let $\delta_{1}(\ell):=\FDiff{1}$. Then for $k\geq1$,
\begin{equation}\label{eq:Ak.deltak}
    \Ak{k}(T,\ell)
    = \delta_{k}(\ell)\cdots\delta_{1}(\ell)\left(\fil\Bigl(\frac{\ell}{T}\Bigl)\right).
\end{equation}
By induction using \eqref{eq:Delta.uell.nuell} and \eqref{eq:Ak.deltak},
$\Ak{k}$ with $k\geq1$ can be written as
\begin{equation}\label{eq:est. Ak-2}
    \Ak{k}(T,\ell)
    = \sum_{i=1}^{k}R_{-(2k-1-i)}(\ell) \:\FDiff{i}\:\fil\Bigl(\frac{\ell}{T}\Bigr),
\end{equation}
where $R_{-j}(\ell)$, $k-1\leq j\leq 2k-2$, is a rational function of
$\ell$ with degree\footnote{Let $R(t)$ be a rational polynomial taking the
form $R(t)=p(t)/q(t)$, where $p(t)$ and $q(t)$ are polynomials with
$q\neq0$. The \emph{degree} of $R(t)$ is $\deg(R):=\deg(p)-\deg(q)$.}
$\deg(R_{-j})\leq -j$ and hence
\begin{eqnarray}\label{eq:est. Ak-3}
    R_{-j}(\ell)= \bigo{d,k}{\ell^{-j}}.
\end{eqnarray}

For $\fil\in \CkR$ and $0\leq i\leq k\leq \fis$, we have by induction the following
integral representation of the finite difference
$\FDiff{i}\:\fil(\frac{\ell}{T})$:
\begin{equation*}
\FDiff{i}\:\fil\Bigl(\frac{\ell}{T}\Bigr) = \int_{0}^{\frac{1}{T}}\IntD{u_{1}}\cdots \int_{0}^{\frac{1}{T}}\fil^{(i)}\left(\frac{\ell}{T}+u_{1}+\cdots+u_{i}\right)\IntD{u_{i}}.
\end{equation*}
Since $\fil^{(i)}$ is bounded, for $T_{1}-\fis\leq \ell\leq T_{2}$,
\begin{equation}\label{eq:FDiff.UB}
    \Bigl|\FDiff{i}\:\fil\Bigl(\frac{\ell}{T}\Bigr)\Bigr| \leq c_{i,\fil}\: T^{-i}.
\end{equation}
This together with \eqref{eq:est. Ak-2} and \eqref{eq:est. Ak-3} gives
\eqref{eq:est.Ak}, on noting that $\ell\asymp T$ in \eqref{eq:est.Ak}.
\end{proof}

\begin{proof}[Proof of Theorem~\ref{thm:filter.sph.ker.UB}]
In this proof, let $r:=(d-2)/2$, $T_{1}:=\lceil aT\rceil$ and $T_{2}:=\lfloor 2T\rfloor$. We only need to consider $T$ sufficiently large to ensure that $0\leq T_{1}-\fis \leq T_{2}$. Let $\Jcb{\ell}(t)$, $t\in[-1,1]$, be the Jacobi polynomial of degree $\ell$ for $\alpha,\beta>-1$.
From \cite[Eq.~4.5.3, p.~71]{Szego1975},
\begin{equation}\label{eq:filtered ker-1}
  \sum_{j=0}^{\ell}\frac{(2j+\alpha+\beta+1)\:\Gamma(j+\alpha+\beta+1)}{\Gamma(j+\beta+1)}\Jcb{j}(t)
 = \frac{\Gamma(\ell+\alpha+\beta+2)}{\Gamma(\ell+\beta+1)}\Jcb[{\alpha+1,\beta}]{\ell}(t),
\end{equation}
and by \cite[Eq.~4.1.1, p.~58]{Szego1975},
    $\Jcb{\ell}(1)={\ell+\alpha\choose \ell}$.
Then we find using \eqref{eq:dim.sph.harmon} and \eqref{eq:normalised.Gegenbauer} that
\begin{align}\label{eq:filtered ker-2}
  \vdh{T,\fil}(\cos\theta)
  &= \sum_{\ell=0}^{\infty}\fil\Bigl(\frac{\ell}{T}\Bigr) \: Z(d,\ell) \:\NGegen{\ell}(\cos\theta)\notag\\
  &= \frac{\Gamma(\frac{d}{2})}{\Gamma(d)} 
  \sum_{\ell=0}^{\infty}\fil\Bigl(\frac{\ell}{T}\Bigr)\: \frac{(2\ell+2r+1)\Gamma(\ell+2r+1)}{\Gamma(\ell+r+1)} \:\Jcb[r,r]{\ell}(\cos\theta)\notag\\
  &= \frac{\Gamma(\frac{d}{2})}{\Gamma(d)}
  \sum_{\ell=T_{1}-\fis}^{T_{2}}\hspace{-2mm} \Ak{\fis}(T,\ell)\:
  \frac{\Gamma(\ell+2r+\fis+1)}{\Gamma(\ell+r+1)} \:\Jcb[r+\fis,r]{\ell}(\cos\theta),
\end{align}
where the last equality uses \eqref{eq:filtered ker-1} and summation by
parts $\fis$ times, and $\Ak{\kappa}(T,\ell)$ is given by \eqref{eq:def. A_k(n,l)}.

From \cite[Eq.~7.32.5, Eq.~4.1.3]{Szego1975} or \cite[Eq.~B.1.7, p.~416]{DaXu2013}, for arbitrary $\alpha,\beta > -1$,
\begin{equation}\label{eq:Jacobi loc est}
  \bigl|\Jcb{\ell}(\cos\theta)\bigr|\leq  \frac{c_{\alpha,\beta}\: \ell^{-\frac{1}{2}}}{(\ell^{-1}+\theta)^{\alpha+\frac{1}{2}}(\ell^{-1}+\pi-\theta)^{\beta+\frac{1}{2}}},\;\; 0\leq \theta\leq \pi.
\end{equation}

Applying Lemma~\ref{lm:est. Ak} with \eqref{eq:filtered ker-2} and \eqref{eq:Jacobi loc est} gives (bearing in mind that $r=(d-2)/2$)
\begin{align*}
|\vdh{T,\fil}(\cos\theta)|
    &\leq
  c_{d,\fis}\sum_{\ell=T_{1}-\fis}^{T_{2}}|\Ak{\kappa}(T,\ell)|\:\ell^{r+\fis}\times\frac{ \ell^{-\frac{1}{2}}}{(\ell^{-1}+\theta)^{r+\fis+\frac{1}{2}}(\ell^{-1}+\pi-\theta)^{r+\frac{1}{2}}}\nonumber\\
  &\leq c_{d,\fil,\fis}\sum_{\ell=T_{1}-\fis}^{T_{2}}\frac{T^{-(2\fis-1)}\:\ell^{\frac{d}{2}+\fis-\frac{3}{2}}}{(\ell^{-1}+\theta)^{\fis+\frac{d-1}{2}}(\ell^{-1}+\pi-\theta)^{\frac{d-1}{2}}}.
\end{align*}
From this and $T_{1}\asymp T\asymp T_{2}$ together with $T_{2}-T_{1}\asymp T$, for $\theta \in [0,\pi/2]$ we have
\begin{equation*}
  |\vdh{T,\fil}(\cos\theta)|
    \leq c_{d,\fil,\fis}\sum_{\ell=T_{1}-\fis}^{T_{2}} T^{-(2\fis-1)} \frac{T^{\frac{d}{2}+\fis-\frac{3}{2}}}{(T^{-1}+\theta)^{\fis+\frac{d-1}{2}}}
  \leq c_{d,\fil,\fis}\:\frac{T^{d}}{(1+T\theta)^{\fis+\frac{d-1}{2}}};
\end{equation*}
while for $\theta\in [\pi/2,\pi]$,
\begin{equation*}
  |\vdh{T,\fil}(\cos\theta)|\leq c_{d,\fil,\fis}\sum_{\ell=T_{1}-\fis}^{T_{2}} T^{-(2\fis-1)} T^{d+\fis-2}
  \leq c_{d,\fil,\fis}\:T^{d-\fis} \leq c_{d,\fil,\fis} \:\frac{T^{d}}{(1+T\theta)^{\fis}}.
\end{equation*}
The estimates for the above two cases imply \eqref{eq:filter.sph.ker.UB},
thus completing the proof.
\end{proof}

\begin{proof}[Proof of Theorem~\ref{thm:filter.sph.ker.L1norm.UB}]
In the proof, let $T_{1}:=\lceil aT\rceil$ and $T_{2}:=\lfloor 2T\rfloor$. We only need to consider $T$ sufficiently large to ensure that $0\leq T_{1}-\fis \leq T_{2}$.
Using summation by parts $\fis$ times, for $t\in [-1,1]$,
\begin{equation}\label{eq:fiker.ceker}
  \vdh{T,\fil}(t)
  = \sum_{\ell=0}^{\infty}\fil\Bigl(\frac{\ell}{T}\Bigr) \: Z(d,\ell) \:\NGegen{\ell}(t)
  = \sum_{\ell=T_{1}-\fis}^{T_{2}}\left(\FDiff{\fis}\fil\Bigl(\frac{\ell}{T}\Bigr)\right) \cecoe[\fis-1]{\ell}\: \ceker[\fis-1]{\ell}(t),
\end{equation}
where for $k\in\Zp$, $\cecoe{\ell}:=\frac{\Gamma(\ell+k+1)}{\Gamma(\ell+1)\Gamma(k+1)}\asymp_{k} \ell^{k}$ where we used an asymptotic estimate of Gamma function, see e.g. \cite[Eq.~5.11.13, Eq.~5.11.15]{NIST:DLMF},
and
\begin{equation*}
    \ceker{\ell}(t) := \frac{1}{\cecoe{\ell}}\sum_{j=0}^{\ell} \cecoe{\ell-j} \: Z(d,j)\:\NGegen{j}(t), \quad t\in[-1,1]
\end{equation*}
is the Kogbetliantz (Ces\`{a}ro) kernel of order $k$, see e.g. \cite[Section~2.3]{WaLi2006} and \cite[Eq.~2.1.5]{BeBuPa1968}.
Berens et al. \cite[Theorem~2.1.1]{BeBuPa1968}, see also \cite[Theorem~2.3.10]{WaLi2006}, proved that for $\fis\ge \floor{\frac{d+3}{2}}$ and $\PT{x}\in\sph{d}$,
\begin{equation*}
    \normb{\ceker[\fis-1]{\ell}(\PT{x}\cdot\cdot)}{\Lp{1}{d}} \le c_{d,\fis,\fil}.
\end{equation*}
This together with \eqref{eq:fiker.ceker} and \eqref{eq:FDiff.UB} gives
\begin{align*}
  \normb{\vdh{T,\fil}(\PT{x}\cdot\cdot)}{\Lp{1}{d}}
  &\le \sum_{\ell=T_{1}-\fis}^{T_{2}} \Bigl|\FDiff{\fis}\fil\Bigl(\frac{\ell}{T}\Bigr)\Bigr|\: \bigl|\cecoe[\fis-1]{\ell}\bigr|\: \normb{\ceker[\fis-1]{\ell}(\PT{x}\cdot\cdot)}{\Lp{1}{d}}\\
  &\le \sum_{\ell=T_{1}-\fis}^{T_{2}} T^{\fis}\: \ell^{\fis-1}\: c_{d,\fis,\fil}\\
  &\le c_{d,\fis,\fil},
\end{align*}
thus completing the proof.
\end{proof}

\section*{Acknowledgements}
The authors thank Heping Wang for pointing out that the optimal rate should hold in Theorem~\ref{thm:filter.hyper.error.W^s.p}.


\bibliography{DiscreteNeedlet}

\end{document}